\documentclass[reqno]{amsart}
\usepackage{amsmath}
\usepackage{amssymb}
\usepackage{amsfonts}
\usepackage{amsthm}
\usepackage{color}
\usepackage{stmaryrd}

\newtheorem{theorem}{Theorem}[section]
\newtheorem{lemma}[theorem]{Lemma}
\newtheorem{proposition}[theorem]{Proposition}

\newcommand{\nr}{{\mathcal{N}^{1+\frac{2}{r},r}}}

\begin{document}

\title{On the fractional regularity for degenerate equations with $(p,q)$-growth}
\author{Lu\'is H. de Miranda}\thanks{The first author was partially supported by FAPDF-Brazil, grant 4749.25.27523.07072015.}
\address[Lu\'\i s Henrique de Miranda]{Departamento de Matem\'atica, Universidade de Bras\'ilia, Campus Universit\'ario Darcy Ribeiro, Bras\'\i lia-DF, 70910-900, Brazil}
 \email{demiranda@unb.br}

\author{Adilson E. Presoto}\thanks{The second author was partially supported by FAPESP-Brazil, grant 2015/20831-7.}
\address[Adilson E. Presoto]{Departamento de Matem\'atica, Universidade Federal de S\~ao Carlos,13565-905, S\~ao Carlos -
SP, Brazil}
 \email{presoto@dm.ufscar.br}



\begin{abstract}
This paper addresses the gain of global fractional regularity in Nikolskii spaces for solutions of a class of quasilinear degenerate equations with $(p,q)$-growth. Indeed, we investigate the effects of the datum on the derivatives of order greater than one of the solutions of the $(p,q)$-Laplacian operator, under Dirichlet's boundary conditions. As it turns out, even in the absence of the so-called Lavrentiev phenomenon and without variations on the order of ellipticity of the equations,  the fractional regularity of these solutions ramifies depending on the interplay between the growth parameters $p$, $q$ and the data. Indeed, we are going to exploit the absence of this phenomenon in order to prove the validity up to the boundary of some regularity results, which are known to hold locally, and as well provide new fractional regularity for the associated solutions. In turn, there are obtained certain global regularity results by means of the combination between new a priori estimates and approximations of the differential operators, whereas the nonstandard boundary terms are handled by means of a careful choice for the local frame.
\end{abstract}

\maketitle
\section{Introduction}


The present work is devoted to the investigation of the fractional regularity of solutions to the following class of degenerate elliptic equations

 \begin{equation}
  \label{eq1}
    \tag{$\mathcal{D}$}
  \begin{cases}
   -\alpha\Delta_p u- \beta \Delta_q u  & = f  \text{ in } \Omega\\
   \nonumber
   \quad \quad \quad \quad \quad \ \quad u&=0    \text{ on } \partial \Omega
   \end{cases}
   \end{equation}
 where  $q\geq p>2$, $\alpha > 0$, $\beta\geq0$, and $\Omega \subset \mathbb{R}^N$ is an open bounded domain of class $C^{2,1}$. Indeed, our aim is to describe the effects of the parameters $\alpha$ and $\beta$, which control the ellipticity of \eqref{eq1}, and also the interference of the interplay between $p$, $q$, and the order of integrability of $f$ on the spatial derivatives of order greater than one of the solutions to this class of equations, the well-known $(p,q)$-Laplacian operator.

In the past years, the investigation of regularity properties for solutions of quasilinear equations for which the ellipticity has nonstandard growth, i.e., involving two different powers, has been widely addressed in the literature on the area, mostly after the remarkable contributions of Marcellini and Lieberman, among others, for instance see \cite{lie2,mar} and the references therein. After that, several authors have helped to provide results on maximal regularity, higher differentiability, Calder\'{o}n-Zygmund estimates, among other topics, of the associated solutions for large a class of equations with nonstandard ellipticity which generalize \eqref{eq1} in several distinct ways, including the case where the ellipticity depends on the spatial variables, what may cause the occurrence of the so-called Lavrentiev phenomenon, see \cite{zhikov}. In order to remark some of the works addressing these matters, once again with no intention of being complete, we refer the reader to \cite{akm,bcm,byo,ckn1,ckn2,com1,com3,com2,elm,min}, or the survey \cite{min2}, and also the references therein. However, we must stress that there are still few results concerning the fractional regularity of the solutions to such equations, see \cite{bgop,cpa} for instance.

Thus, our purpose is to revisit \eqref{eq1} and provide some new results on the global fractional regularity of its solutions. Indeed, in order to illustrate certain ideas, let us consider momentarily the case where $\beta=\beta(x)$ is Lipschitz continuous, i.e.,
 \begin{equation}
 \label{eq11}
 \tag{$\mathcal{D}_x$}
  -\alpha\Delta_p u-  \mbox{div}\big(\beta(x)|\nabla u|^{q-2}\nabla u\big)  = f  \text{ in } \Omega,
  \end{equation}
  where $f$ belongs to some space which is imbedded in $L^s(\Omega)$.
 As it turned out, one of the interesting points of problems like  \eqref{eq11} is that there may have a certain loss of information on the regularity of its solutions due to the instability on the degree of ellipticity of the equation, depending on the behavior of $\beta(x)$ and most of all, if $p$ and $q$ are not close enough. Indeed, there is a distortion when we contrast the case $\alpha=0$ and $\alpha>0$ which is caused, in part, by the Lavrentiev phenomenon, see \cite{com1,elm}. Moreover, we remark that under suitable conditions for $f$, $p$ and $q$, for instance for the case $f\equiv 0$, if \[p<q< p+\frac{p}{N},\] and \(\beta(x)\) is locally Lipschitz, the solutions are locally regular. In addition, this bound could be relaxed for \[p<q\leq p+1\]  when the solutions of \eqref{eq11} are assumed to be locally bounded, what allows the validity of very interesting results concerning existence and regularity of solutions, see \cite{bcm,com2}.
Despite that, up to now, even for the case where $\alpha$ and $\beta$ are constants, the current literature provides few results regarding  global fractional regularity of the solutions, and most of all, almost none information on the interference of the interplay among the data, $p, q$ and $s$, on this sort of regularity up to the boundary. Thence, since the interaction between these parameters must change the character of the differential operator which is being considered, our purpose is to make these effects explicit in terms of the fractional regularity and foremost, in terms of global a priori estimates of the associated solutions. Thereby, it is our main goal to investigate these problems and to provide precise answers to the effect of the correlation between these parameters on the fractional regularity of the solutions, what is done by means of explicit global a priori estimates on adequate functional spaces, see Theorems \ref{theorem1} and \ref{theorem2} below. Indeed, we are going to revisit the most classic case for $(p,q)$- ellipticity, namely \eqref{eq1}, and provide new results for the regularity of its solutions. This will include interesting phenomena regarding some kind of ramification for the fractional differentiability of these solutions even in the case where there is no change in the order of ellipticity of the problem and in the absence of Lavrentiev phenomenon, see Theorem \ref{theorem3} below.

Further, we point out that the investigation of higher differentiability of the solutions for this sort of equations usually is associated with the behavior of the vector fields
\[v_p=|\nabla u|^{p-2}\nabla u \mbox{ or } v_q=|\nabla u|^{q-2}\nabla u\]
or specifically, whether $v_p$ or $v_q$ belong to $W^{1,2}(\Omega)$, which is actually sharp in some cases, see for instance \cite{min} and the references therein.
However, as it turned out, depending on the regularity of the data, there are other vector fields which cause the appearance of additional fractional differentiability.  In the present work, we describe this additional regularity for the case of the $(p,q)$-Laplacian and also provide new a priori estimates for its solutions. Actually,  the fractional regularity obtained in our main results is a consequence of a thorough analysis of the vector fields
\[v_{r}=|\nabla u|^{r-2}\nabla u \]
where $r=r(p,q,f,\alpha,\beta)>0$ will be given below and the associated regularity ramifies depending on the interaction between the data. Surprisingly, even in the case where $\beta>0$, the pure $(p,q)$-growth, the fractional regularity of the solutions splits along certain Nikolskii spaces according to the interconnection of these parameters beneath certain threshold  cases, see Theorem \ref{theorem3}.

Finally, if on one hand, we do not know if the regularity provided in our main results is sharp, any sort of improvement on it would depend on providing new information for the $v_r$ vector fields.


\subsection{Notations, assumptions and main results}
First of all, we must introduce the following class of exponents
\begin{equation}
\label{exponents}\tag{$\mathcal{E}$}
r_i=
\begin{cases}

&s(p-2)+2, \mbox{ for }i=1
\\
&s(q-2)+2, \mbox{ for }i=2
\\
&s(p-2)+2+q-p,\mbox{ for }i=3
\\
&s(q-2)+2+p-q, \mbox{ for }i=4
\end{cases}
\end{equation}
where $s>2$.
As it turns out, each possibility describes certain degrees of regularity for the solutions of \eqref{eq1}. The specific choice of $r_i$  will naturally depend on the interplay between the parameters $\alpha, \beta, p, q$, and $f$.

Our approach to regularity will involve functional spaces of fractional order of differentiability. As a matter of fact, despite that there are nowadays countless possibilities to describe derivatives of fractional order, in the present discussion we will consider mostly the case of the Nikolskii  spaces, which are a particular case of the more general Besov class. For the sake of completeness we briefly describe these spaces below, and for further information, refer the interested reader to \cite{grisvard,kufner,leoni}, 
where the relying theory is discussed thoroughly.

The choice for the spaces is motivated mostly by the nature of the estimates which appear on the subsequent calculations, so that is some sense, they are natural by the point of view of the differential operators considered. From now on, given an open set $V\subset \mathbb{R}^M$, $M=N$ or $N-1$, we denote by
\[\mathcal{N}^{1+\frac{2}{r_i},r_i}(V)=\{u\in W^{1,r_i}(V): \llbracket u \rrbracket_{\mathcal{N}^{1+\delta,r_i}(V)}<+\infty\}\]
the Nikolskii space of order $1+\frac{2}{r_i}$. We stress that $\llbracket u \rrbracket_{\mathcal{N}^{1+\frac{2}{r_i},r_i}}$ is the so-called Gagliardo-Nikolskii seminorm or, simply, the Nikolskii seminorm, which, for a given $h\in \mathbb{R}^N\setminus \{0\},$ is defined by
\begin{equation}
\nonumber
\llbracket u \rrbracket_{\mathcal{N}^{1+\frac{2}{r_i},r_i}(V)}=\bigg(\sup_{|h|>0}\int_{V_{|h|}} \dfrac{|\nabla u(x+h)-\nabla u(x)|^{r_i}}{|h|^{2}}\bigg)^{\frac{1}{r_i}},
\end{equation}
where $V_{|h|}=\{ x\in V: \mbox{dist}(x,\partial V)\geq |h|\}$.

Further, we remark that in the present work each one of such spaces is going to be endowed with the norms described below
\begin{equation}
\nonumber
\begin{cases}
&\|u\|_{\mathcal{N}^{1+\frac{2}{r_i},r_i}(V)}= \|u\|_{W^{1,p}(V)}+\llbracket u \rrbracket_{\mathcal{N}^{1+\frac{2}{r_i},r_i}(V)}, \mbox{ for } i=1 \mbox{ or }4
\\
&\|u\|_{\mathcal{N}^{1+\frac{2}{r_i},r_i}(V)}= \|u\|_{W^{1,q}(V)}+\llbracket u \rrbracket_{\mathcal{N}^{1+\frac{2}{r_i},r_i}(V)},\mbox{ for } i=2 \mbox{ or }3.
\end{cases}
\end{equation}
In turn, these choices are motivated by the peculiar set up of \eqref{eq1}, and are more suitable for our purposes. Anyway, by simple interpolation arguments and under our basic assumptions it is clear that the latter norms are equivalent to the standard ones, i.e., where $W^{1,p}$ or $W^{1,q}$ are replaced by $W^{1,r_i}$

Complementarily, we also consider the following fractional spaces
\[W^{1+\delta,r}(V)=\{u\in W^{1,r}(V): \llbracket u \rrbracket_{W^{1+\delta,r}(V)}<+\infty\}\]
for \[1\leq r <+\infty
, 0<\delta<1,\]
the standard Sobolev-Slobodeckii spaces of order $1+\delta$, which are endowed with norms given by
\[\|u\|_{W^{1+\delta,r}(V)}= \|u\|_{W^{1,r}(V)}+\llbracket u \rrbracket_{W^{1+\delta,r}(V)}.\]
In the last case,
\begin{equation}
\label{slobodeckiinorm}\tag{$\mathcal{G}$}
\nonumber
\llbracket u \rrbracket_{W^{1+\delta,r}(V)}=\bigg(\int\int_{V \times V} \dfrac{|\nabla u(x)-\nabla u(y)|^r}{|x-y|^{N+\delta r}}\bigg)^{1/r}
\end{equation}
is the associated Gagliardo-Slobodeckii or Slobodeckii-Sobolev seminorm.

In addition, we stress that along our entire discussion, in order to not overburden the notations and since the difference will be clear from the context, we will  use the same notations for scalars, vectors in $\mathbb{R}^N$ and square matrices in $\mathcal{M}(N\times N)$. For instance, we write  $u \in L^2(\Omega)$,  even when $u$ is a vector field or a square matrix, meaning that all of its components are in $L^2(\Omega)$, while $|u|$ may indicate the absolute value, the Euclidean norm in $\mathbb{R}^N$ or $\mathbb{R}^{N^2}$. Moreover, the norms  on $(L^{r}(\Omega))^N$, $(L^{r}(\Omega))^{N^2}$ and $L^{r}(\Omega)$  will be denoted  as $\|.\|_{L^{r}}$,
and so on.

Finally, let us remark that, from now on, along the text, the dependence of the function spaces on $\Omega$ will be omitted, whenever this does not lead to confusion, while dependence on $\partial \Omega$ or $\overline{\Omega}$ will always be emphasized. Thence $\mathcal{N}^{1+\frac{2}{r},r}$ denotes $\mathcal{N}^{1+\frac{2}{r},r}(\Omega)$ and $W^{1+\delta,r}$ denotes $W^{1+\delta,r}(\Omega)$. Further, we stress that $\epsilon$ will always be assumed to satisfy $\epsilon \in (0,1]$, and $C$ stands for a general positive constant, which may vary from line to line. In turn, the dependence of $C$ on the data will be indicated as $C=C(...)>0$.

\subsection{Basic Assumptions}
From now on, unless otherwise stated, we going to consider the following hypotheses.
\begin{align}
\label{hip1}\tag{$\mathcal{H}_1$} &\Omega \subset \mathbb{R}^N \mbox{ is an open bounded set where } \partial \Omega \in C^{2,1}.
\\
\label{hip2} \tag{$\mathcal{H}_2$} &f \in W^{\sigma,s}(\Omega) \mbox{ where } s \geq 2 \mbox{ and } \sigma > \dfrac{1}{s}.
\\
\label{hip3}\tag{$\mathcal{H}_3$}  &q\geq p> 2, \alpha >0, \mbox{ and } \beta \geq 0.
\end{align}

\subsection{Main Results}
Is this section, there will be provided the main contributions of the present paper.

First, we present a result which provides some contribution for the investigation of the global fractional regularity of solutions  even for the case of the single $p$-Laplacian.
\begin{theorem}
\label{theorem1}

Under the hypotheses \eqref{hip1},\eqref{hip2} and \eqref{hip3} 
  there exists a unique strong solution to \eqref{eq1}, i.e.,
\[u \in \mathcal{N}^{1+\frac{2}{r_1},r_1}(\Omega) \cap W^{1,p}_0(\Omega),\]
where $r_1=s(p-2)+2$, such that
\[-\alpha \Delta_p u+\beta \Delta_q u = f \mbox{ a.e. in } \Omega.\]
Moreover, there exists $C=C(N,p,s,\alpha,\beta,\sigma,\Omega)>0$ for which
\[ \|u\|_{\mathcal{N}^{1+\frac{2}{r_1},r_1}(\Omega)}^{r_1} \leq C \bigg (\|f\|^{\frac{r_1}{p-1}}_{W^{\sigma,s}(\Omega)}+\|f\|^{\frac{r_1}{p-1}}_{L^{p^\prime}(\Omega)}+\|f\|^s_{L^s(\Omega)} +\|f\|^2_{L^s(\Omega)} \bigg) .\]
In particular, since $s\geq 2$
\[ \|u\|_{\mathcal{N}^{1+\frac{2}{r_1},r_1}(\Omega)}^{r_1} \leq C \bigg (\|f\|^{s}_{W^{\sigma,s}(\Omega)}+1\bigg) .\]

\end{theorem}
Now, we give a few comments on some results which are somehow connected to Theorem \ref{theorem1}. In \cite{bgop}, see also \cite{cpa}, the authors obtain fractional regularity for the solutions of a broad class of degenerate equations with nonsmooth coefficients which generalize the $p$-Laplacian case. Indeed, by means of a difference quotient approach and by supposing that the data relies on Besov spaces it is proved that these solutions  belong to a class of Besov spaces, locally see \cite[Theorems 1.1-1.3]{bgop}. Further, among other contributions, in \cite{elm} (see also \cite{com3}) the authors obtain sharp conditions for local higher integrability  and gradient H\"{o}lder continuity whereas, in \cite{bcm}, it is provided  a maximal regularity result for the associated solutions. In addition, in \cite{com1}, Theorem 1.5 and in \cite{com2}, Theorem 3.1, there are proved general fractional regularity results, which in particular hold, locally, for \eqref{eq1} in the case that $f=0$. More recently, in \cite{byo}, Lemma 3.7, it is provided, global fractional regularity also for the homogeneous case, when $\Omega$ is a ball. Moreover, in \cite{ckn1} and \cite{ckn2}, the authors prove sharp local regularity, i.e. in $W^{1,2}_{loc}$, for the vector field $v_p$. We stress that in the latter cases, the results are valid for a large class of equations, which in particular include \eqref{eq1}. Additionally, we remark that for the case where $f$ is a Borel measure, recently Avelin, Kuusi and Mingione have obtained optimal local fractional regularity for the solutions, see \cite{akm}.

Further, we remark that Theorem \ref{theorem1} complements the contributions given in \cite{dmm1} and \cite{dmm2} since we now consider a more general operator and, most of all, we cover Dirichlet boundary conditions, what demands some improvements of the technique which imply modifications on the estimates. Indeed, the boundary data coming from the fractional estimates for $p$-Laplacian-like operators is highly nonlinear and is handled by means of delicate estimates, involving boundary derivatives of order two combined with powers of the gradients on the boundary, see \cite{dmm1} or \cite{dmm2}. As it turns out, Neumann boundary conditions regularize the boundary functionals associated to solutions of Partial Differential Equations in two levels. First by canceling out nonlinear terms which depend on derivatives of order higher than one, and formally, by imposing to some regularity for the boundary derivatives for the solutions. In the Dirichlet case, all of this information is lost, so that in order to obtain the same sort of control on such nonlinear terms it was necessary to modify our approach, what alters some estimates. Actually, we employed the equation itself in order to control the boundary data, what allowed us to obtain a new set of a priori estimates for the Dirichlet case. We stress that, by technical reasons, this is  why we asked for extra regularity of the data $f$, to assure that we have control on the trace of $f$, see Lemma \ref{lemb2} below.

Complementarily, we stress that there are also interesting results concerning higher fractional differentiability to a large class of distinct equations. Without the intention of being complete, for instance, in the case of nonlocal equations, we cite the results obtained in \cite{kms,schi}. Indeed, in \cite{kms}, the authors provide a generalized version of the Gehring Lemma, and then in Theorem 1.1, prove local fractional regularity in terms of Sobolev-Slobodeckii spaces with the order of differentiability greater than one for the solutions of a broad class of equations. In addition, in \cite{schi}, by means of a generalization of the idea of differentiating the equation, the author obtains higher order fractional differentiability for the solutions of the so-called fractional $p$-Laplacian.

The next result exhibits the effect of the interaction between the $p$ and $q$-Laplacian on the regularity of the solutions. 
\begin{theorem}
\label{theorem2}
Under the hypotheses \eqref{hip1},\eqref{hip2} and \eqref{hip3},  suppose that $\beta>0$.
Thus the unique solution of \eqref{eq1} satisfies
\[u \in \mathcal{N}^{1+\frac{2}{r_2},r_2}(\Omega)\cap W^{1,q}_0(\Omega)\]
and
\[ \|u\|_{\mathcal{N}^{1+\frac{2}{r_2},r_2}(\Omega)}^{r_2} \leq C \bigg (\|f\|^{\frac{r_2}{q-1}}_{W^{\sigma,s}(\Omega)}+\|f\|^{\frac{r_2}{q-1}}_{L^{q^\prime}(\Omega)}+\|f\|^s_{L^s(\Omega)} +\|f\|^2_{L^s(\Omega)} \bigg) .\]
where  $C=C(N,p,q,s,\alpha,\beta,\sigma,\Omega)>0$ and $r_2=s(q-2)+2$.
\end{theorem}
 The last result is a counterpart of Theorem \ref{theorem1} since it guarantees a new level of smoothness for the solutions. Indeed, a distinct type of fractional regularity appears when $\beta>0$, namely, $\mathcal{N}^{1+\frac{2}{r_2},r_2}(\Omega)$, which improves the result of Theorem \ref{theorem1} to a different direction, i.e., the regularity evolves to a different level when contrasted to the previous one since, in general, $\mathcal{N}^{1+\frac{2}{r_i},r_i}(\Omega)\not\hookrightarrow\mathcal{N}^{1+\frac{2}{r_j},r_j}(\Omega)$ and $\mathcal{N}^{1+\frac{2}{r_i},r_i}(\Omega) \not\hookrightarrow\mathcal{N}^{1+\frac{2}{r_j},r_j}(\Omega)$ for $i\neq j$.
 Also, remark that in an analogous manner to the $(p,2)$-Laplacian case, for instance, see \cite{bmp2}, Proposition 1, the perturbation of the $p$-Laplacian by a $q$-Laplacian gives some sort of ramification for the global regularity of the solutions.

Finally, our last result provides a measure of the interference in the regularity of the solutions given by the interactions between the parameters $p,q,$ and $s$. Indeed, as it turned out, rather than a genuine replacement of regularity, the appearance of another quasilinear degenerate differential operator gives rise to several degrees of regularity as a counterpart to the $p$-$q$ growth unbalance.

\begin{theorem}
\label{theorem3}
Under the hypotheses \eqref{hip1},\eqref{hip2} and \eqref{hip3},  suppose that $\beta>0$.
\begin{itemize}
\item[(a)] If $s\geq \frac{q+p-4}{p-2}$ then the unique solution of \eqref{eq1} satisfies
\[u \in \mathcal{N}^{1+\frac{2}{r_3},r_3}(\Omega)\]
and
\[ \|u\|_{\mathcal{N}^{1+\frac{2}{r_3},r_3}(\Omega)}^{r_3} \leq C \bigg (\|f\|^{\frac{r_3}{q-1}}_{W^{\sigma,s}(\Omega)}+\|f\|^{\frac{r_3}{q-1}}_{L^{q^\prime}(\Omega)}+\|f\|^\tau_{L^s(\Omega)} +\|f\|^2_{L^s(\Omega)} \bigg) .\]
where  $r_3=r_1+q-p$ and $\tau=\frac{s(p-2)+q-p}{q-2}$.
\item[(b)]
If \[p<q<p+1 \mbox{ and }s\geq 1 + \dfrac{1}{1+p-q},\] suppose in addition that $f \in L^\rho(\Omega)$, where $\rho=\frac{s(q-2)+p-q}{p-2}$. Then the unique of \eqref{eq1} satisfies
\[u \in \mathcal{N}^{1+\frac{2}{r_4},r_4}(\Omega)\]
and
\[ \|u\|_{\mathcal{N}^{1+\frac{2}{r_4},r_4}(\Omega)}^{r_4} \leq C \bigg (\|f\|^{\frac{r_4}{p-1}}_{W^{\sigma,s}(\Omega)}+\|f\|^{\frac{r_4}{p-1}}_{L^{p^\prime}(\Omega)}+\|f\|^\rho_{L^\rho(\Omega)} +\|f\|^2_{L^s(\Omega)} \bigg) .\]
where  $r_4=r_2+p-q.$
\end{itemize}
In the latter inequalities, $C=C(N,p,q,s,\alpha,\beta,\sigma,\Omega)>0$.
\end{theorem}

To the best of our knowledge, Theorem \ref{theorem3} is entirely new, since it provides information on the existence of two additional levels of smoothness for the solutions of \eqref{eq1} which arise from the $p$-$q$ unbalance and its interaction with $s$ for the case of fixed ellipticity. It seems to describe some sort of ramification phenomenon, at least regarding fractional regularity, since after the threshold parameters
\[\frac{q+p-4}{p-2} \mbox { or } 1 + \dfrac{1}{1+p-q}\]
the last one for the case where $p<q<p+1$, the solutions gain regularity. Is is interesting that despite that \eqref{eq1} does not possess the Lavrentiev phenomenon, still under the assumption that $p$ and $q$ are close enough, new regularity and a priori estimates appear. This is consistent with the results related to more general operators, see for instance \cite{byo} or \cite{com1,com3,com2}. Moreover, remark that under our assumptions, $2\leq \tau \leq s$ whereas $2\leq s \leq \rho$, and $\rho=\tau=s$ if $p=q$. Thus, once again the case $q>p$ gives the impression of possessing a certain degree of ramification in what regards the regularity of the solutions. Remark that this is also reflected in the terms which depend of $f$ in the last a priori estimates.

The present paper is organized as follows. Section \ref{preliminary} concerns the basic tools on Fractional Regularity, namely the interplay between Fractional Spaces and degenerate operators, combined with certain approximation results. In Section \ref{boundaryestimates}, we address certain highly nonlinear boundary terms which are linked to the energy bounds considered along the text. Further, Section \ref{energyestimates} is devoted to the discussion of a new set of a priori estimates which are the basis of our entire method and part of our main contributions, and finally, in Section \ref{mainresults} we provide our ultimate proofs which guarantee the validity of the main results of this paper.

\section{Preliminary results}
\label{preliminary}
For the sake of clarity, we exhibit below some basic tools of Fractional Regularity which are going to be exploited in the present paper. Most of the results below are merely direct adaptations of well-known results, which will play important roles in the present paper. Namely, we will present some of the basic theory on Fractional Spaces which we employ to investigate the fractional regularity of the solutions, what includes obtaining adequate approximate smooth solutions. For the convenience of the reader and/or since some of these results are interesting in their own right, we exhibit some of its proofs and omit others.

\subsection{Basic Tools for Fractional Regularity}

In what follows, for the sake of convenience and further reference, we state a basic, however, important inequality. Its proof follows by completely analogous arguments of those used in \cite{els} or \cite{dmm2}. The most important character of this inequality it is its invariability with respect to $\epsilon$.  

\begin{lemma}
\label{lemp4}
Let \(r\geq 2\). Then there exists a constant \(C=C(r)\) such that, for all \(\epsilon>0\), we have
 \begin{equation*}
  |U-V|^r\leq C\left|U(|U|^2+\epsilon)^{\frac{r-2}{4}}-V(|V|^2+\epsilon)^{\frac{r-2}{4}}\right|^2, \forall \  U \mbox { and } V \in \mathbb{R}^N.
 \end{equation*}
\end{lemma}

Now, we are in the position to prove certain standard imbedding inequalities for Nikolskii spaces in terms of nonlinearities which are naturally related to the solutions of \eqref{eq1}, see \cite{dmm1}. On behalf of the reader's convenience, we present its proof.

\begin{lemma}
\label{lemp5}
Consider $r\geq p$. There exists $C=C(N,p,\Omega)>0$ for which
\begin{equation}
\nonumber \|u\|^r_{\mathcal{N}^{1+\frac{2}{r},r}}\leq C\bigg(\bigg\|\nabla\bigg(\nabla u(|\nabla u|^2+\epsilon)^{\frac{r-2}{4}}\bigg)\bigg\|_{L^2}^2+\|u\|_{W^{1,p}}^r\bigg)
\end{equation}
for all $u\in W^{3,\tau}$, where $\tau>N$.
If $r\geq q$ then the same result holds if we replace $W^{1,p}$ by $W^{1,q}$.
\end{lemma}
\begin{proof}
Consider the Nikolskii seminorm
\[\llbracket u\rrbracket^r_{\mathcal{N}^{1+\frac{2}{r},r}}=\sup_{h}\int_{\Omega_{|h|}}\dfrac{|\nabla u (x+h)-\nabla u (x)|^r}{|h|^2},\]
where $\Omega_{|h|}=\{x\in \Omega: \mbox{d}(x,\partial\Omega)>|h|\}$. By Lemma \ref{lemp4}, it is clear that
\begin{align}
\label{eq2lemp5}
\llbracket u \rrbracket^r_{\nr} &\leq C \sup_{h}\int_{\Omega_{|h|}} \dfrac{|v(x+h)-v(x)|^2}{|h|^2}
\\
\nonumber
&\leq C \int_\Omega |\nabla v|^2,
\end{align}
for $v=\nabla u \big(\big(|\nabla u|^2+\epsilon\big)^{\frac{r-2}{4}}\big) \in W^{1,2}(\Omega)$.

Moreover, remark that by employing basic interpolation arguments, given $\eta>0$, there exists $K=K(N,p,r,\Omega)>0$ such that
\begin{equation}
\label{eq3lemp5}
\|u\|^r_{W^{1,r}}\leq \eta \|u\|^r_{\nr}+K\|u\|^r_{W^{1,p}}.
\end{equation}

Thus, the result follows by setting $\eta=1/2^{r+1}$ and by combining \eqref{eq2lemp5} and \eqref{eq3lemp5}.

\end{proof}

\subsection{Approximations}
\label{approximations}
In this subsection, we briefly review certain known results on the existence of {\it smooth} solutions to approximate versions of \eqref{eq1}. 

For this purpose, we are going to make use of the following approximations. Indeed, let us introduce
\begin{equation}
\label{eq1bd}
\Delta_p^\epsilon u = \mbox{div} \bigl ( (|\nabla u|^2+\epsilon)^{(p-2)/2} \nabla u \bigr ), \quad p > 2,
\end{equation}
a perturbation of the $p$-Laplacian, where $\epsilon \in (0,1]$ stands for the degeneracy parameter and $\Delta_q^\epsilon u$ is defined in an analogous manner.
Moreover, set
\begin{equation}
\label{eq2bd}
f_\epsilon \in C^{\infty}(\overline{\Omega}) \mbox{ such that }f_\epsilon \to f \mbox{ in }W^{\sigma,s}(\Omega), \mbox{ as }\epsilon \to 0.
\end{equation}

In this fashion, we consider below quasilinear approximation of our Dirichlet \eqref{eq1} 
boundary value problem.
As a matter of fact, let
\begin{equation}
  \label{eq1ap}
    \tag{$\mathcal{D}_\epsilon$}
  \begin{cases}
   -\alpha\Delta_p^\epsilon u- \beta \Delta_q^\epsilon u  & = f_\epsilon  \text{ in } \Omega\\
   \nonumber
   \quad \quad \quad \quad \quad \ \quad u&=0    \text{ on } \partial \Omega
   \end{cases}
   \end{equation}
the nondegenerate approximations of \eqref{eq1} 
where $\Delta_p^\epsilon$, $\Delta_q^\epsilon$ and $f_\epsilon$ were defined in \eqref{eq1bd} and \eqref{eq2bd}.

There is a vast literature on the existence of smooth solutions to \eqref{eq1ap} and related problems. For instance, in a chronological order, we cite \cite{lady}, specially Chapters 4 and 10, \cite{grisvard}, for instance Chapter 2, \cite{lie1}, \cite{lie2}, and \cite{gtr}, specially Chapters 7. However, since in the present case we deal with a $(p,q)$-nonlinear term, for the reader convenience we provide the details of the proofs for existence results in $W^{3,\tau}(\Omega)$, where $\tau>N$, which are appropriate for the present approach.

\begin{lemma}\label{approximatesolutions}Consider $\epsilon \in (0,1]$. Under the hypotheses \eqref{hip1} and \eqref{hip3}, given $\tau>N$,
there exists a unique $u\in W^{3,\tau}$ solution of \eqref{eq1ap}. Moreover,
\[ \|u\|_{W^{3,\tau}}\leq C\| f_\epsilon\|_{L^\infty},\]
where $C=C(N,\alpha,\beta,\epsilon,\tau,\Omega)>0$
\end{lemma}
\begin{proof}
As the proof of the uniqueness is readily obtained by well-know monotonicity arguments, it is left to the reader.

Thus, we can focus at existence and regularity of solutions. For this, we introduce some convenient auxiliary functions.

Indeed, set $A:\mathbb{R}^N\to \mathbb{R}^N$, $B:\Omega \to \mathbb{R}$ and $D:\mathbb{R}^N\to \mathbb{R}$ given by
\begin{align}
\nonumber
A(v)&=D(v) v
\\
\nonumber
B(x)&=-f_\epsilon(x)
\\
\label{eq1lemap1}
D(v)&= \alpha(|\nabla v|^2+\epsilon )^{(p-2)/2}+\beta(|\nabla v|^2+\epsilon )^{(q-2)/2}.
\end{align}

Then, \eqref{eq1ap} can be rewritten as
\begin{equation*}
  \begin{cases}
   -\mbox{div} \bigg(A\big(\nabla u\big)\bigg)+B(x)  & = 0  \text{ in } \Omega\\
   \nonumber
   \quad \quad \quad \quad \quad \quad \quad \quad \ \quad u&=0    \text{ on } \partial \Omega.
   \end{cases}
   \end{equation*}
where $A \in C^\infty(\mathbb{R}^N,\mathbb{R}^N)$, $B \in C^2(\overline{\Omega})$ and $D \in C^\infty(\mathbb{R}^N)$.

Suppose that $\beta>0$.

Now, it is clear that, cf. 10.2 and 10.5, p. 259-60, 10.23, p. 271, and 15.82, p. 381,  in \cite{gtr},  $A$ and $B$ satisfy the following structure conditions,
\begin{align*}
\sum_{i,j=1}^N \dfrac{1}{2}\bigg(\dfrac{\partial A_j}{\partial v_i}+\dfrac{\partial A_i}{\partial v_j}\bigg)\xi_i \xi_j &\geq D(v)|\xi|^2 \ \forall \xi \in \mathbb{R}^N,
\\
\nonumber
v\cdot A(v)&\geq \beta |v|^q,
\\
\nonumber
|v|^{q-2}&\leq \mbox{O}\big(D(v)\big), \mbox{ since } \beta >0,
\\
\nonumber
|D_v A|&\leq \mbox{O}\big(|v|^{q-2}\big),
\\
\nonumber
B&= \mbox{O}\big(|v|^{q}\big).
\end{align*}
By the above inequalities, and by Thm 15.11, p. 381, in \cite{gtr}, given $\gamma \in (0,1)$, there exists $u\in C^{2,\gamma}(\overline{\Omega})$ satisfying \eqref{eq1ap}.

In the sequel, let us consider
\[ Lv= \mbox{ div}\bigg( F(\nabla u)\bigg) \nabla v)+f.\]
It is clear that this (linear) operator satisfies the hypotheses  of Thm. 9.19 p. 243 in \cite{gtr} for $k=1$.
Therefore, $u\in W^{3,\tau}$, for $\tau>N$.
For the case $\beta=0$ it is enough to consider $(p,\alpha)$ instead of $(q,\beta)$ in the latter arguments.

\end{proof}
\section{Boundary estimates}
\label{boundaryestimates}
In this section, we deal with certain nonstandard nonlinear boundary terms, which in the present case, arouse as a consequence of the Dirichlet boundary conditions coupled to the $(p,q)$-Laplacian.
The  case of Neumann boundary conditions, which is more simple, was treated in \cite{dmm1} and \cite{dmm2}, for the $p$-Laplacian, for $f\in L^2(\Omega)$. 
 Essentially, when the normal derivatives of the solutions are null on the boundary, the terms depending on second order derivatives cancel out so that it is possible the apply trace theorems in order to control the nonlinear gradient terms on the boundary.
  However, in the Dirichlet case there is no cancelation, what makes the estimates more delicate. As it turned out, for Dirichlet boundary conditions,  it is not possible to reduce the second order derivatives on the boundary, so that  improvements on the technique and further assumptions on the data had to be employed.
Basically, we have observed that in the case of Dirichlet boundary conditions, the nonlinear terms could be rewritten in terms of the $(p,q)$-Laplacian. This brings the possibility to use deeper information of the data on the boundary, in order to compensate the lack of information of the derivatives of solutions on $\partial \Omega$, when contrasting to the Neumann case.

For this purpose, we are going to make use of a family of the so-called Moving Frames, considered at every point $x\in \partial \Omega$ and given by tangent and normal vector of $\partial \Omega$.
This means that given any $x\in \partial \Omega$ there exists \[\mathfrak{B}_x=\{\tau_1(x),\cdots, \tau_{N-1}(x), \eta(x)\}\] an orthonormal basis for $\mathbb{R}^N$, for which $T_x \partial \Omega = Span \{ \tau_1(x),\cdots, \tau_{N-1}(x), \nu(x)\}$,  the tangent plane of $\partial \Omega$ at $x$, and $\nu(x)$ denotes its exterior normal unit vector.
In order to visually simplify the notations which we employ, we are going to omit the dependence of the Moving Frame on $x\in \partial \Omega$

For the reader's convenience and the sake of clarity, we provide the details below.
Indeed, given $x\in \partial \Omega$, consider $W\subset \mathbb{R}^N$, an open neighborhood of $x$ and set $V=W\cap \partial \Omega$, chosen sufficiently small so that there exists $\{ \tau_1,\cdots, \tau_{N-1}\}$, a local geodesic frame of $\partial \Omega \cap V$. In particular, for $\nu$ the exterior normal unit vector of $\partial \Omega$, we have that
\[ \{\tau_1,\cdots, \tau_N, \nu\}\]
is a local orthonormal frame in $W$. This in particular implies that at every $y\in W$, $\mathfrak{B}_y=\{\tau_1(y),\cdots, \tau_{N-1}(y),\nu(y)\}$ is an orthonormal basis of $\mathbb{R}^N$.

One of the main advantages of the above family of orthonormal vectors of is that at the points  $y \in \partial \Omega \cap V$, we can choose an orthonormal basis $\mathbb{R}^N$, $\mathfrak{B}_y$ such that the differentiation with respect to these directions fits more properly to the standard boundary conditions
\[ u=0 \mbox{ or } \dfrac{\partial u}{\partial \nu}=0 \mbox{ on } \partial \Omega\]
than other choices for bases, such as the canonical basis $\{e_1,\cdots, e_N\}.$

For simplicity, we drop the dependence of $\mathfrak{B}_y$ on the point $y$.


In this fashion, given any vector field defined in $W$, $v:W\rightarrow \mathbb{R}^N$, there holds
\[ v(y)= \sum_{i=1}^{N-1} v_i(y) \tau_i + v_\nu(y) \nu, y \in  W\]
where $\tau_i=\tau_i(y)$ and $\nu_i=\nu_i(y)$.

Moreover, given $\phi \in C^1(\overline{\Omega})$, remark that $\nabla \phi$ could be written with respect to $B$, so that
\[ \nabla \phi = \sum_{i=1}^{N-1} \dfrac{\partial \phi}{\partial \tau_i} \tau_i + \dfrac{\partial \phi}{\partial \nu}\nu.\]

As a matter of fact, by considering $\eta \in \mathbb{R}^N$, and , we arrive at
\begin{align*}
 \nabla \phi(y)\cdot \eta &= d_y\phi(\eta)=\sum_{i=1}^{N-1} \eta_i d_y\phi(\tau_i)+\eta_\nu d_y\phi(\nu)\\
 &=\bigg(\sum_{i=1}^{N-1} \dfrac{\partial \phi}{\partial \tau_i}(y)\tau_i+\dfrac{\partial \phi}{\partial \nu}(y)\nu\bigg)\cdot \eta,
\end{align*}
since $\displaystyle \eta= \sum_{i=1}^{N-1} \eta_i \tau_i + \eta_\nu \nu$.

Finally, we stress that the latter construction of the Moving Frame $\mathfrak{B}$ could be done, locally, for every $x\in \partial \Omega$, so that indeed depend on the neighborhood $W$. However, since we are going to rewrite our equations in a way that they are independent of $W$, we may disregard the influence of $W$ in our boundary estimates. Thus, by using the above representation, from now on, we denote
\[ \nabla_T\phi = \sum_{i=1}^{N-1} \dfrac{\partial \phi}{\partial \tau_i}\tau_i \mbox{ on } \partial \Omega,\]
and so on.

Now, we are going to present some brief lemmata which we employ in order to control the highly nonlinear terms appearing on the boundary.

The following result is essential for our purposes, despite its simple proof, which for the  convenience of the reader, we provide the details.

\begin{lemma}
\label{lemb1}
Given $u \in C^1(\overline{\Omega})$ such that $u=0$ on $\partial \Omega$. Then $\nabla_T u =0$ on $\partial \Omega$
\end{lemma}
\begin{proof}
Indeed, given $x\in \partial \Omega$, set
\[
\sigma: [0,T] \rightarrow \partial \Omega \]
such that
\begin{align*}
&\sigma(t_0)=x \\
&\sigma^\prime(t_0)=\tau_i(x), \mbox{ where } t_0 \in (0,T).
\end{align*}

Since, $u\big|_{\partial \Omega}=0$, then
\[ d_{\sigma(t)} u (\sigma ^\prime (t))= \dfrac{d}{dt}u(\sigma(t))=0.\]
Thus, by fixing $t=t_0$ we obtain
\[\dfrac{\partial u }{\partial \tau_i}= d_x u (\tau_i)=0\]
and the result follows.
\end{proof}

At this point, before proceeding to the estimation of the boundary terms, we must introduce some notation employed on the rest of this subsection.
For $u\in C^2(\overline{\Omega})$, we set
\begin{align}
\label{eq2boundary}
G_u(x)&= (|\nabla u(x)|^2+\epsilon )^{t/2}\dfrac{\partial u}{\partial \nu}(x)\Delta u(x)
\\
\nonumber&+ t(|\nabla u(x)|^2+\epsilon)^{(t-2)/2}\Delta u(x) \dfrac{\partial u}{\partial \nu}(x)|\nabla u(x)|^2,
\end{align}
where $x\in \overline{\Omega}$ and $t\in \mathbb{R}$.
Moreover, by recalling the definition of $D$ in Lemma \ref{approximatesolutions}, see \eqref{eq1lemap1}, we consider
\begin{align}
\label{eq1boundary}
D_u(x)&= \big(\alpha(|\nabla u(x)|^2+\epsilon )^{(p-2)/2}+\beta(|\nabla u(x)|^2+\epsilon )^{(q-2)/2}\big).
\end{align}

For simplicity of exposition, whenever it is clear from the context, in what follows we  drop the dependence of these functions on $x$.

Below, we will provide some lemmata which will be used to control boundary terms for the Dirichlet case. We start by the following technical lemma.

\begin{lemma}
\label{lemb2}
Given $u\in C^2(\overline{\Omega})$, and $t\in \mathbb{R}$, suppose that $u=0$ on $\partial \Omega$. Then, there exists $C=C(p,t)>0$, such that
\[ |D_u G_u|\leq C | \alpha \Delta_p^\epsilon u+\beta \Delta^\epsilon_q u| (|\nabla u|^2+\epsilon)^{(t+1)/2} \mbox { on } \partial \Omega.\]
\end{lemma}
\begin{proof}
First, observe that for $x\in \partial \Omega$
\begin{align}
\label{eq1lemb2}
\Delta_p^\epsilon u &= \mbox{div}\big((|\nabla u|^2+\epsilon)^{(p-2)/2}\nabla u \big)
\\
\nonumber
&= \sum_{i=1}^{N-1} \dfrac{\partial}{\partial \tau_i} \bigg((|\nabla u|^2+\epsilon)^{(p-2)/2} \dfrac{\partial u}{\partial \tau_i}\bigg)+\dfrac{\partial}{\partial \nu} \bigg((|\nabla u|^2+\epsilon)^{(p-2)/2} \dfrac{\partial u}{\partial \nu}\bigg)
\end{align}

However, by applying Lemma \ref{lemb1} we have that
\[\dfrac{\partial u}{\partial \tau_i}=0 \mbox{ on }\partial \Omega, \mbox{ for } i=1,\cdots, N-1.\]
In particular, for the same reason,
\[\dfrac{\partial^2 u}{\partial \tau_i \partial \tau_j}=0 \mbox{ on } \partial \Omega, \mbox{ for } i,j=1,\cdots, N-1.\]

Thus, by direct calculations, we are led to
\begin{align}
\label{eq2lemb2}
&\sum_{i=1}^{N-1} \dfrac{\partial }{\partial \tau_i}\bigg( (|\nabla u|^2+\epsilon)^{(p-2)/2}\dfrac{\partial u}{\partial \tau_i}\bigg)=  \sum_{i=1}^{N-1}  (|\nabla u|^2+\epsilon)^{(p-2)/2}\dfrac{\partial^2 u}{\partial \tau_i^2}
\\
\nonumber
&+(p-2)\sum_{i=1}^{N-1}(|\nabla u |^2+\epsilon )^{(p-4)/2}\bigg\{\sum_{j=1}^{N-1}\dfrac{\partial u}{\partial \tau_j}\dfrac{\partial^2 u}{\partial\tau_i\partial\tau_j}\dfrac{\partial u}{\partial \tau_i}+\dfrac{\partial u}{\partial \nu}\dfrac{\partial^2 u}{\partial\tau_i\partial\nu}\dfrac{\partial u}{\partial \tau_i}\bigg\}
\\
\nonumber
&=0 \mbox{ on } \partial \Omega.
\end{align}

Moreover, in an analogous manner

\begin{align}
\label{eq3lemb2}
&\dfrac{\partial}{\partial \nu} \bigg( (|\nabla u|^2+\epsilon)^{(p-2)/2}\dfrac{\partial u}{\partial \nu}\bigg)=(|\nabla u|^2+\epsilon)^{(p-2)/2}\dfrac{\partial^2 u}{\partial \nu^2}
\\
\nonumber
&+(p-2)(|\nabla u|^2+\epsilon)^{(p-4)/2}\bigg\{\sum_{i=1}^{N-1}\dfrac{\partial u}{\partial \tau_i}\dfrac{\partial^2u}{\partial\nu\partial \tau_i}\dfrac{\partial u}{\partial\nu}+\bigg(\dfrac{\partial u}{\partial \tau_i}\bigg)^2\dfrac{\partial^2u}{\partial\nu^2}\bigg\}
\\
\nonumber
&=(|\nabla u|^2+\epsilon)^{(p-2)/2}\Delta u +(p-2)(|\nabla u|^2+\epsilon)^{(p-4)/2}|\nabla u|^2\Delta u.
\end{align}
Hence, from \eqref{eq1lemb2}-\eqref{eq3lemb2}, one gets that
\[ \Delta_p^\epsilon u = (|\nabla u|^2+\epsilon)^{(p-2)/2}\Delta u +(p-2)(|\nabla u|^2+\epsilon)^{(p-4)/2}|\nabla u|^2\Delta u.
\]
Then, by using the above representation for $p$ and $q$, there follows that
\begin{align}
\nonumber
\alpha \Delta_p^\epsilon u +\beta \Delta_q^\epsilon u  = & \  \alpha(|\nabla u|^2+\epsilon)^{(p-2)/2}\Delta u
\\\nonumber
& +\alpha(p-2)(|\nabla u|^2+\epsilon)^{(p-4)/2}|\nabla u|^2\Delta u
\\\nonumber
&+ \beta(|\nabla u|^2+\epsilon)^{(q-2)/2}\Delta u
\\
&+\beta(q-2)(|\nabla u|^2+\epsilon)^{(q-4)/2}|\nabla u|^2\Delta u.
\nonumber
\end{align}

In addition, by considering
\[ C=\max \bigg\{ 1, \dfrac{|t|}{p-2}\bigg\},\]

we arrive at
\begin{align}
\nonumber
&\bigg (\alpha(|\nabla u|^2+\epsilon)^{(p-4)/2}+\beta (|\nabla u|^2+\epsilon)^{(q-4)/2}\bigg)\bigg|(|\nabla u|^2+\epsilon)^2+\epsilon|\nabla u|^2\bigg|
\\\nonumber
&\leq C\alpha(|\nabla u|^2+\epsilon)^{(p-4)/2}\bigg((|\nabla u|^2+\epsilon)^2+(p-2)|\nabla u|^2\bigg)
\\\label{eq5lemb2}
 &+ C\beta(|\nabla u|^2+\epsilon)^{(q-4)/2}\bigg((|\nabla u|^2+\epsilon)^2+(q-2)|\nabla u|^2\bigg)
\\
\nonumber
&=C\bigg|\alpha(|\nabla u|^2+\epsilon)^{(p-4)/2}\bigg((|\nabla u|^2+\epsilon)^2+(p-2)|\nabla u|^2\bigg)
\\\nonumber
 &+\beta(|\nabla u|^2+\epsilon)^{(q-4)/2}\bigg((|\nabla u|^2+\epsilon)^2+(q-2)|\nabla u|^2\bigg)\bigg|.
\end{align}
In this way, by the above inequality and the definition of $F$ and $G$, given $x\in \partial \Omega$

\begin{align}
\nonumber
\bigg | D_u(x)G_u(x)\bigg| &\leq  \bigg| \alpha (|\nabla u|^2+\epsilon)^{(p-2)/2}+\alpha t (|\nabla u|^2+\epsilon )^{(p-4)/4}
\\
\label{eq6lemb2}
&+\beta (|\nabla u|^2+\epsilon)^{(q-2)/2}+\beta t (|\nabla u|^2+\epsilon )^{(q-4)/4}\bigg||\Delta u|
(|\nabla u|^2+\epsilon)^{t/2}\bigg|\dfrac{\partial u}{\partial \nu }\bigg| .
\end{align}

Therefore, by combining \eqref{eq5lemb2} and \eqref{eq6lemb2}, we arrive at
\begin{align*}
\bigg | D_u(x)G_u(x)\bigg| &\leq C\bigg|\alpha \Delta_p^\epsilon u +\beta \Delta_q^\epsilon u \bigg|\bigg|\dfrac{\partial u}{\partial \nu }\bigg| (|\nabla u|^2+\epsilon)^{t/2},
\end{align*}
on $\partial \Omega$.
\end{proof}

With the above results, we are able to prove a very important estimate which allows us to handle the nonlinear boundary terms in the Dirichlet case.
In contrast to \cite{dmm1} and \cite{dmm2}, now, in order to guarantee such control we had to ask for more information on $f$ on $\partial \Omega$.
In some sense, we compensate the lack of regularity of $u$ on $ \partial \Omega$ in comparison to the Neumann case with additional information on the trace of $f$.
\begin{lemma}
\label{lemb3}
Under our basic assumptions, consider that $t\in (-1,+\infty)$. Given $u \in C^2(\overline{\Omega})$ a classic solution of \eqref{eq1ap} there holds that
\begin{align}
\nonumber
&\bigg | \sum_{i,j=1}^N \int_{\partial \Omega }  \bigg(\alpha (|\nabla u |^2+\epsilon)^{(p-2)/2}u_{x_i}
\\
\nonumber
&+\beta (|\nabla u |^2+\epsilon)^{(q-2)/2}u_{x_i}\bigg)\big (|\nabla u |^2+\epsilon)^{t/2}u_{x_j}\big)_{x_i}\nu_j dS\bigg|
\\
\label{eq11lemb3}
&\leq C \|f_\epsilon\|_{L^s(\partial\Omega)}\bigg( \|\nabla u\|^{t+1}_{L^{s^\prime(t+1)}(\partial \Omega)}+\epsilon^{(t+1)/2}\bigg),
\end{align}
where $C=C(p,s,t,\partial \Omega)>0$.
\end{lemma}
\begin{proof}
First of all, we stress that given $\omega \in \mathbb{R}^N$, by considering an extension of the normal unit outer vector to a small neighborhood of $\partial \Omega$, there follows that
\begin{equation}
\label{eq2lemb3}
\dfrac{\partial \nu(x)}{\partial \omega } \cdot \nu(x)=0 \quad \forall x \in \partial \Omega,
\end{equation}
since $\nu (x) \cdot \nu (x) =1$.

Moreover, in other to visually simplify our calculations, we consider
\begin{align}
\nonumber
H_u &= \sum_{i,j=1}^N  \bigg(\alpha (|\nabla u |^2+\epsilon)^{(p-2)/2}u_{x_i}+\beta (|\nabla u |^2+\epsilon)^{(q-2)/2}u_{x_i}\bigg)\big (|\nabla u |^2+\epsilon)^{t/2}u_{x_j}\big)_{x_i}\nu_j
\\
\nonumber
&= \bigg(D_u(x)\nabla u \cdot \nabla \bigg)
\\
\nonumber
&\circ \bigg( (|\nabla u|^2+\epsilon)^{t/2}\nabla u\bigg)\cdot \nu,
\\
&\label{eq0lemb3}
\end{align}
where $D_u$ was defined in \eqref{eq1boundary}.

Since we are interested in $H_u$'s behavior on $\partial \Omega$, it is more convenient to make use of the moving frame described above, $\mathfrak{B}_x$ for every $x\in \partial \Omega$. Thus, by switching the  frame we arrive at
\begin{align*}
H_u& = D_u \times \bigg\{ \sum_{i,j=1}^{N-1} \dfrac{\partial u }{\partial \tau_i} \dfrac{\partial}{\partial \tau_i} \bigg( (|\nabla u|^2+\epsilon)^{t/2}\dfrac{\partial u}{\partial \tau_j} \tau_j\bigg)\cdot \nu
\\
&+\sum_{i=1}^{N-1} \dfrac{\partial u }{\partial \tau_i} \dfrac{\partial}{\partial \tau_i} \bigg( (|\nabla u|^2+\epsilon)^{t/2}\dfrac{\partial u}{\partial \nu} \nu\bigg)\cdot \nu+\sum_{j=1}^{N-1} \dfrac{\partial u }{\partial \nu} \dfrac{\partial}{\partial \nu} \bigg( (|\nabla u|^2+\epsilon)^{t/2}\dfrac{\partial u}{\partial \tau_j} \tau_j\bigg)\cdot \nu
\\
&+ \dfrac{\partial u }{\partial \nu} \dfrac{\partial}{\partial \nu} \bigg( (|\nabla u|^2+\epsilon)^{t/2}\dfrac{\partial u}{\partial \nu} \nu\bigg)\cdot \nu\bigg\}, \mbox{ on } \partial \Omega.
\end{align*}

Thus, by recalling that $\tau_i\cdot \nu =0$, for $i=1,\cdots,N-1$, and also, by combining the latter identity with \eqref{eq2lemb3}, one gets
\begin{align}
\nonumber
H_u& = D_u \times \bigg\{ \sum_{i,j=1}^{N-1} \dfrac{\partial u }{\partial \tau_i}  (|\nabla u|^2+\epsilon)^{t/2}\dfrac{\partial u}{\partial \tau_j} \dfrac{\partial \tau_j}{\partial \tau_i}\cdot \nu
\\
\nonumber
&+\epsilon\sum_{i,j=1}^{N-1} \dfrac{\partial u }{\partial \tau_i} (|\nabla u|^2+\epsilon)^{(t-2)/2}\dfrac{\partial^2u}{\partial \tau_i\partial \tau_j}\dfrac{\partial u}{\partial \tau_j}\dfrac{\partial u}{ \partial\nu}
\\
\nonumber
&+\epsilon\sum_{i=1}^{N-1} \dfrac{\partial u }{\partial \tau_i} (|\nabla u|^2+\epsilon)^{(t-2)/2}\dfrac{\partial^2 u}{\partial \tau_i \partial \nu} \bigg(\dfrac{\partial u}{\partial \nu}\bigg)^2+\sum_{i=1}^{N-1} \dfrac{\partial u }{\partial \tau_i}  (|\nabla u|^2+\epsilon)^{t/2}\dfrac{\partial^2 u}{\partial \tau_i\partial \nu}
\\
\nonumber
& +\sum_{j=1}^{N-1} \dfrac{\partial u }{\partial \nu}  (|\nabla u|^2+\epsilon)^{t/2}\dfrac{\partial u}{\partial \tau_j} \dfrac{\partial \tau_j}{\partial \nu}\cdot \nu+\epsilon\sum_{i=1}^{N-1}\bigg(\dfrac{\partial u }{\partial \nu}\bigg)^2 (|\nabla u|^2+\epsilon)^{(t-2)/2}\dfrac{\partial^2 u}{\partial \nu\partial \tau_i}\dfrac{\partial u}{\partial \tau_i}
\\
\label{eq3lemb3}
&+ t \bigg(\dfrac{\partial u }{\partial \nu}\bigg)^3  (|\nabla u|^2
+\epsilon)^{(t-2)/2}\dfrac{\partial^2 u}{\partial \nu^2}+ \dfrac{\partial u }{\partial \nu}  (|\nabla u|^2+\epsilon)^{t/2}\dfrac{\partial^2 u}{\partial \nu^2} \bigg\},
\end{align}
on $\partial \Omega$.

The boundary condition $u = 0$ on $\partial \Omega$, according to Lemma \ref{lemb1}, guarantees that
\[ \dfrac{\partial u}{\partial\tau_i}=\dfrac{\partial u}{\partial \tau_i \partial\tau_j}=0 \quad \forall i,j=1,\cdots,N-1,\]
which in particular implies that $\Delta u = \dfrac{\partial ^2u}{\partial \nu^2}$ on $\partial \Omega$.

In this fashion, by recalling the definition of $G_u$ in \eqref{eq2boundary}, it turns out that \eqref{eq3lemb3} becomes
\begin{align*}
\nonumber
H_u&= D_u  \bigg(t \bigg(\dfrac{\partial u }{\partial \nu}\bigg)^3  (|\nabla u|^2+\epsilon)^{(t-2)/2}\dfrac{\partial^2 u}{\partial \nu^2}+ \dfrac{\partial u }{\partial \nu}  (|\nabla u|^2+\epsilon)^{t/2}\dfrac{\partial^2 u}{\partial \nu^2} \bigg)
\\
&=D_u  \bigg(t \bigg(\dfrac{\partial u }{\partial \nu}\bigg)^3  (|\nabla u|^2+\epsilon)^{(t-2)/2}\Delta u+ \dfrac{\partial u }{\partial \nu}  (|\nabla u|^2+\epsilon)^{t/2}\Delta u \bigg)
\\
&=D_u  G_u, \mbox{ on } \partial \Omega.
\end{align*}

Thus,  by Lemma \ref{lemb2}
\begin{align}
\nonumber
|H_u| &\leq C | \alpha \Delta_p^\epsilon u+\beta \Delta^\epsilon_q u| (|\nabla u|^2+\epsilon)^{(t+1)/2}
\\
\label{eq4lemb3}
&\leq C|f_\epsilon|(|\nabla u|^2+\epsilon)^{(t+1)/2},
\end{align}
on $\partial \Omega$.

Therefore, by combining \eqref{eq0lemb3} and \eqref{eq4lemb3}, after integrating over $\partial \Omega$ and by applying H\"{o}lder's inequality for $s$ and $s^\prime$, we obtain \eqref{eq11lemb3}.
\end{proof}

Finally, we exhibit a set of imbeddings  which will allow us to compensate the highly nonlinear boundary terms appearing on the latter estimates by means of the fractional a priori bounds of the solutions themselves. Despite that the proof is well-know, we provide the details for the sake of completeness.
\begin{lemma}
\label{lemb4}
Under our basic assumptions, consider $r_i$ given by \eqref{exponents} and $t_i$ given by \eqref{auxiliaryexponents} and \eqref{auxiliaryexponents2}. The imbedding
\[ \mathcal{N}^{1+\frac{2}{r_i},r_i}(\Omega) \hookrightarrow\hookrightarrow W^{1,s^\prime(t_i+1)}(\partial \Omega)\]
holds true for $i=1,2$ and $3$.
Moreover, if in addition
\[p<q<p+1 \mbox{ and }s\geq 1 + \dfrac{1}{1+p-q}\]
then the latter imbedding also holds for $i=4$.
\end{lemma}
\begin{proof}
 First remark that for all $t>0$, $\eta>0$, sufficiently small,
 \[
  W^{1+\frac{1}{s^\prime(t+1)}+\eta,s^\prime(t+1)}(\Omega) \hookrightarrow W^{1,s^\prime(t_i+1)}(\partial \Omega),
 \]
 see for instance \cite{grisvard}, Thm. 1.5.1.2, p. 37.

 Moreover, for a given $r>2$, it is well-know that
  \[
  \mathcal{N}^{1+\frac{2}{r},r}(\Omega) \hookrightarrow\hookrightarrow W^{1+\frac{2}{r}-\eta,r}( \Omega)
 \]
 see \cite{knees}, Lemma 2.1.

 Thus, it is enough to prove that, for a sufficiently small $\eta>0$, there holds
  \[
 W^{1+\frac{2}{r_i}-\eta,r_i}( \Omega) \hookrightarrow W^{1+\frac{1}{s^\prime(t_i+1)}+\eta,s^\prime(t_i+1)}(\Omega)
 ,\]
 where $t_i=r_i-p$, for $i=1,4$, and $t_i=r_i-q$, for $i=2,3$.

 Remark that in order to guarantee the validity of the latter imbedding,  by Thm. 1.4.4.1, \cite{grisvard} p. 17, we only need to verify that
 \[
 \dfrac{N-2}{N-1} < \dfrac{r_i}{s^\prime(t_i+1)}=\dfrac{r_i(s-1)}{s(t_i+1)}.
 \]
 Since
\[
 \dfrac{N-2}{N-1}<1
 \]
 the result follows by the choices of $r_i$, $t_i$  and by isolating $s$ in terms of $r_i$ in each case.
\end{proof}
\section{Energy estimates}
\label{energyestimates}
In this section, we provide the core energy estimates for the regularity of solutions of \eqref{eq1}.
These estimates are the natural generalizations of the ones obtained in \cite{dmm1} and \cite{dmm2}, with two main differences. First and most of all, we now have to handle more delicate boundary data, since here we do not consider Neumann boundary conditions. We stress that these boundary conditions are on the basis of the energy estimates obtained in \cite{dmm1} or \cite{dmm2} since they allow the simplification of the nonstandard boundary terms which appear on the estimates. In the present case, the technique needs to be modified, so that by means of the equation itself we are able to handle the boundary integral terms but now in the case of Dirichlet boundary conditions. Second, we have to deal with a nonlinear source of $(p,q)$-type, which brings an extra unbalance bias to the estimates. However, we actually use this additional unbalance to obtain extra estimates.

For the sake of clarity, we split the proof of the energy estimates along some brief lemmata. We stress that despite most of the  results of the present section hold under weaker assumptions on $u$, the proofs below are stated for \(u\in W^{3,\tau}\), where $\tau>N$. This is done for the sake of clarity, since $W^{3,\tau}$ is the class of solutions which we are able to obtain existence results for the approximate problems in Section \ref{approximations}.

Let us start with a very simple result, which, however, represents the basic idea in order to obtain higher regularity of our solutions.

\begin{lemma}
\label{leme1}
Under our basic assumptions, given \(u\in W^{3,\tau}\), where $\tau>N$, consider
 the following functionals
 \begin{align}
 \nonumber
 & I_{a,\gamma,t}=\sum_{j=1}^N\gamma(a-2+t)\int_{\Omega}\left(|\nabla u|^2+\epsilon\right)^{\frac{a-4+t}{2}}\left(\sum_{i=1}^Nu_{x_ix_j}u_{x_i}\right)^2
  \\
    \nonumber
   &+\gamma(a-2)t\int_{\Omega}\left(|\nabla u|^2+\epsilon\right)^\frac{a-6+t}{2}\left(\sum_{i,j}^Nu_{x_i}u_{x_ix_j}\right)^2
   \\
   \nonumber
    &+\gamma\int_{\Omega}\left(|\nabla u|^2+\epsilon\right)^\frac{a-2+t}{2}|D^2u|^2
    \\
    \label{efunc}
    &\equiv I^1_{a,\gamma,t}+I^2_{a,\gamma,t}+I^3_{a,\gamma,t},
 \end{align}
 and
\begin{align}
\label{befunc}
F_t&=&\sum_{i,j=1}^N\int_{\partial \Omega}\left[\alpha\left(|\nabla u|^2+\epsilon\right)^\frac{p-2}{2}u_{x_i}+\beta\left(|\nabla u|^2
+\epsilon\right)^\frac{q-2}{2}u_{x_i}\right]
\\\nonumber
&&\cdot \left[ \left(|\nabla u|^2+\epsilon\right)^\frac{t}{2}u_{x_j}\right]_{x_i}\nu_j\text{ds},
 \end{align}
where \(t \in \mathbb{R}\).
  Then, there holds that
  \begin{align*}
  & \int_\Omega\left[\alpha\left(|\nabla u|^2+\epsilon\right)^\frac{p-2}{2}\nabla u+\beta\left(|\nabla u|^2+\epsilon\right)^\frac{q-2}{2}\nabla u\right]\cdot \nabla \left(-\Delta^{\epsilon}_{t+2}u\right)
\\
&= I_{p,\alpha,t}+I_{q,\beta,t}-F_t.
\end{align*}
  \end{lemma}
\begin{proof}
 We first note that by integrating by parts we have
  \begin{align}
  \nonumber    &\sum_{i,j=1}^N\int_{\Omega}-\alpha\left(|\nabla u|^2+\epsilon\right)^\frac{p-2}{2}u_{x_i}\left[\left(|\nabla u|^2+\epsilon\right)^\frac{t}{2}u_{x_j}\right]_{x_j,x_i}\\\nonumber
    &  =\sum_{i,j=1}^N\int_{\Omega}\alpha\left[\left(|\nabla u|^2+\epsilon\right)^\frac{p-2}{2}u_{x_i}\right]_{x_j}\left[\left(|\nabla u|^2+\epsilon\right)^\frac{t}{2}u_{x_j}\right]_{x_i}\\
    \label{eq1leme1}
    &  -\sum_{i,j=1}^N\int_{\partial \Omega}\alpha\left(|\nabla u|^2+\epsilon\right)^\frac{p-2}{2}u_{x_i}\left[\left(|\nabla u|^2+\epsilon\right)^\frac{t}{2}u_{x_j}\right]_{x_i}\nu_j\text{ds}.
 \end{align}
where \(\nu\) is the outward unitary normal vector to \(\partial \Omega\).

Moreover, from the Chain Rule, it follows that
 \begin{align}
 \nonumber
   &\sum_{i,j}^N\int_{\Omega}\left[\alpha\left(|\nabla u|^2+\epsilon\right)^\frac{p-2}{2}u_{x_i}\right]_{x_j}\left[\left(|\nabla u|^2+\epsilon\right)^\frac{t}{2}u_{x_j}\right]_{x_i}\\ \nonumber
   & =\int_{\Omega}\alpha \left(|\nabla u|^2+\epsilon\right)^\frac{p-2+t}{2}|D^2u|^2\\ \nonumber
   & +\alpha(p-2)\sum_{i,j,k=1}^N\int_\Omega \left(|\nabla u|^2+\epsilon\right)^\frac{p-4+t}{2}u_{x_i}u_{x_k}u_{x_kx_j}u_{x_jx_i}\\ \nonumber
   & +\alpha t\sum_{i,j,k=1}^N\int_\Omega \left(|\nabla u|^2+\epsilon\right)^\frac{p-4+t}{2}u_{x_ix_j}u_{x_k}u_{x_kx_i}u_{x_j}\\ \nonumber
  & +\alpha(p-2) t\sum_{i,j,k,l=1}^N\int_\Omega \left(|\nabla u|^2+\epsilon\right)^\frac{p-6+t}{2}u_{x_i}u_{x_k}u_{x_kx_j}u_{x_j}u_{x_l}u_{x_lx_i}\\ \nonumber
  & =\alpha\int_\Omega \left(|\nabla u|^2+\epsilon\right)^\frac{p-2+t}{2}|D^2u|^2\\ \nonumber
  &+\alpha(p-2+t)\sum_{j=1}^N\int_\Omega\left(|\nabla u|^2+\epsilon\right)^\frac{p-4+t}{2}\left(\sum_{i=1}^Nu_{x_i}u_{x_jx_i}\right)^2\\ \nonumber
  &+\alpha(p-2)t\int_\Omega\left(|\nabla u|^2+\epsilon\right)^\frac{p-6+t}{2}\left(\sum_{i,j=1}^Nu_{x_i}u_{x_ix_j}u_{x_j}\right)^2\\
  \label{eq2leme1}
   & =I_{p,\alpha,t}^3+I_{p,\alpha,t}^1+I_{p,\alpha,t}^2.
 \end{align}
Therefore, by combining \eqref{eq1leme1} and \eqref{eq2leme1} with a straightforward adaption of \eqref{eq2leme1} for \(q\) and \(\beta\) instead of \(p\) and \(\alpha\), the result follows.
\end{proof}

Our goal now is to control each of the integral terms of the last lemma separately and somehow correlate them to the imbedding estimate given by Lemma \ref{lemp5}. For this purpose, we have to investigate the behavior of each of these terms in what regards to their dependence on $t$. The positivity of the second one follows clearly from the choice for the range of \(p,q,\) and \(t\).

\begin{lemma}
\label{leme2}
 Under our basic assumptions, let  \(u\in W^{3,\tau}\), and \(t\geq 0\),  where $\tau>N$. If $(a,\gamma)=(p,\alpha)$ or $(q,\beta)$, then
  \begin{equation}
  \label{eq3leme2}
   I^2_{p,\alpha,t}\geq 0 \quad \text{and}\quad I^2_{q,\beta,t}\geq 0.
  \end{equation}
 \end{lemma}
Before proceeding to the main estimates of Section \ref{energyestimates}, we introduce some notation and exhibit an algebraic inequality which will be useful to our purposes.
\begin{lemma}
\label{leme3}
 Given \(r\geq 2\) and \(u\in W^{3,\tau}\), where $\tau>N$, consider
 \begin{equation*}
  S_r=\int_{\Omega}\left|\nabla \left(\nabla u|\nabla u|^2+\epsilon\right)^\frac{r-2}{4}\right|^2.
 \end{equation*}
Then, there holds that
 \begin{align}
 \nonumber
  S_r&\leq \int_\Omega \left(|\nabla u|^2+\epsilon\right)^\frac{r-2}{2}|D^2u|^2\\
    \label{eq3obs1}
     &+\frac{(r-2)(r+2)}{4}\sum_{i=1}^N\int_\Omega \left(|\nabla u|^2+\epsilon\right)^\frac{r-4}{2}\left(\sum_{j=1}^Nu_{x_j}u_{x_ix_j}\right)^2.
  \end{align}
\end{lemma}
 \begin{proof}
 First, we claim that
  \begin{align*}
  S_r&=\int_\Omega \left(|\nabla u|^2+\epsilon\right)^\frac{r-2}{2}|D^2u|^2
  \\
  &+\sum_{i=1}^N\int_\Omega (r-2)\left(|\nabla u|^2+\epsilon\right)^\frac{r-2}{4}\left(\sum_{j=1}^Nu_{x_j}u_{x_ix_j}\right)^2\\
    &+\sum_{i=1}^N\frac{(r-2)^2}{4}\int_\Omega |\nabla u|^2\left(|\nabla u|^2+\epsilon\right)^\frac{r-6}{2}\left(\sum_{j=1}^Nu_{x_j}u_{x_ix_j}\right)^2.
 \end{align*}
 Indeed, observe that
  \begin{align*}
  &\left[\left(|\nabla u|^2+\epsilon\right)^\frac{r-2}{4}u_{x_j}\right]_{x_i}
  \\
  &=\left(|\nabla u|^2+\epsilon\right)^\frac{r-2}{4}u_{x_ix_j}+\frac{r-2}{2}\left(|\nabla u|^2+\epsilon\right)^\frac{r-6}{4}\sum_{l=1}^N u_{x_l}u_{x_ix_l}u_{x_j}.
 \end{align*}
Thus, it is straightforward to check that
 \begin{align*}
  &\left|\nabla \left(\nabla u(|\nabla u|^2+\epsilon\right)^\frac{r-2}{4}\right|^2
                              =\left(|\nabla u|^2+\epsilon\right)^\frac{r-2}{2}|D^2u|^2\\
                              &+\sum_{i=1}^N(r-2)\left(|\nabla u|^2+\epsilon\right)^\frac{r-4}{2}\left(\sum_{j=1}^N u_{x_j}u_{x_ix_j}\right)^2\\
                              &+\sum_{i=1}^N\frac{(r-2)^2}{4}|\nabla u|^2\left(|\nabla u|^2+\epsilon\right)^\frac{r-6}{2}\left(\sum_{j=1}^Nu_{x_j}u_{x_ix_j}\right)^2.
\end{align*}
At last, for the proof of \eqref{eq3obs1}, it is enough to recall that
\begin{equation*}
   (r-2)+\frac{(r-2)^2}{4}=\frac{(r-2)(r+2)}{4}.\qedhere
  \end{equation*}
\end{proof}
%


%
We are now in position to exhibit the energy estimates which guarantee fractional regularity for the gradient of the solutions of \eqref{eq1}.
For this, in the same fashion of the definition of the exponents $r_i$, see \eqref{exponents}, we consider the following auxiliary exponents
\begin{equation}
\label{auxiliaryexponents}
t_1=r_1-p \mbox{ and } t_2=r_2-q.
\end{equation}
Sometimes, in order to keep track of the associated regularity exponent we will also denote
\begin{equation}
\label{auxiliaryexponents2}
t_3=r_3-q \mbox{ and } t_4=r_4-p.
\end{equation}

Since our differential operator is a $(p,q)$-Laplacian, while contrasted to the estimates obtained in \cite{dmm1} and \cite{dmm2} there are new levels of regularity provided by the interplay between $p$ and $q$. For the sake of clarity, we exploit these possibilities separately. In the next lemma, we concentrate in the regularity provided by a test function which is associated to the $q$-Laplacian part of our differential operator. 
\begin{proposition}
 \label{prope1}
 Given $u\in W^{3,\tau}$, where $\tau>N$, suppose that \eqref{hip1}-\eqref{hip3} hold.
Then, if \(t_1=s(p-2)+2-p\), there exists \(C=C(\alpha,\beta,p,q,s)>0\) such that
 \begin{equation*}
  I_{t_1}=I_{p,\alpha,t_1}+I_{q,\beta,t_1}\geq C S_{r_1},
 \end{equation*}
where \(r_1=s(p-2)+2\), see \eqref{exponents} and $I_{r_1}$ is given by Lemma \ref{leme3}.

Moreover, if $\beta>0$ there exists \(C=C(\alpha,\beta,p,q,s)>0\)
for which
\begin{equation*}
I_{t_1}\geq C S_{r_3},
\end{equation*}
for \(r_3=r_1+q-p=s(p-2)+2+q-p\).
\end{proposition}

\begin{proof}
 First, remark that for $t_1=r_1-p$, we have \(t_1\geq 0\). In this way, by Lemma~\ref{leme2}
  \begin{equation}
  \nonumber
    \frac{r_1+2}{4\alpha}I_{p,\alpha,t_1}+C_{\beta,1} I_{q,\beta,t_1}\geq  \frac{r_1+2}{4\alpha}I_{p,\alpha,t_1}^1+ \frac{r_1+2}{4\alpha}I_{p,\alpha,t_1}^3+C_{\beta,1} I_{q,\beta,t_1}^1+C_{\beta,1} I_{q,\beta,t_1}^3,
  \end{equation}

where $C_{\beta,1}= 0$ if $\beta=0$ or \(C_{\beta,1}>\max{\left\{1,\frac{r_3+2}{4\beta}\right\}}\) if $\beta>0$.

Clearly, \(\frac{r_j+2}{4}\geq 1\), for \(j=1,3\), so that
 \begin{equation}\nonumber
    \frac{r_1+2}{4\alpha}I_{p,\alpha,t_1}+C_{\beta,1} I_{q,\beta,t_1}\geq  \frac{r_1+2}{4\alpha}I_{p,\alpha,t_1}^1+ \frac{1}{\alpha}I_{p,\alpha,t_1}^3+C_{\beta,2}\frac{r_3+2}{4}I_{q,\beta,t_1}^1+C_{\beta,2} I_{q,\beta,t_1}^3,
  \end{equation}
for $C_{\beta,2} = 0$ when $\beta=0$ and $\frac{1}{\beta}$ when $\beta>0$.

However, by the choice of \(t_1\), it is clear that \(q-2+t_1=r_3-2\). Thus, by \eqref{eq3obs1}
  \begin{align}
   \nonumber
    &C_{\beta,2} \frac{r_3+2}{4}I^1_{q,\beta,t_1}+C_{\beta,2} I^3_{q,\beta,t_1}
    \\
    &=C_\beta\frac{(r_3+2)(r_3-2)}{4}\sum_{j=1}^N\int_\Omega\left(|\nabla u|^2+\epsilon\right)^\frac{r_3-4}{2}\left(\sum_{i=1}^Nu_{x_ix_j}u_{x_i}\right)^2\\
                      \nonumber
                      &+C_{\beta}\int_\Omega\left(|\nabla u|^2+\epsilon\right)^\frac{r_3-2}{2}|D^2u|^2\geq C_{\beta} S_{r_3},
   \end{align}
   where $C_\beta=1$ if $\beta>0$ or $0$ if $\beta=0$.

In an analogous manner, since \(r_1=t_1+p\), one obtains that
 \begin{align}
 \nonumber
  \frac{r_1+2}{4\alpha}I^1_{p,\alpha,t_1}+\frac{1}{\alpha}I_{p,\alpha,t_1}^3
   \nonumber
             &=\frac{(r_1+2)(r_1-2)}{4}\sum_{j=1}^N\int_\Omega\left(|\nabla u|^2+\epsilon\right)^\frac{r_1-4}{2}\left(\sum_{i=1}^Nu_{x_ix_j}u_{x_i}\right)^2\\
    \nonumber
              &+\int_{\Omega}\left(|\nabla u|^2+\epsilon\right)^\frac{r_1-2}{2}|D^2u|^2\\
              \nonumber
              &\geq S_{r_1}.
 \end{align}

Thus, by 
 combining the last three inequalities, there follows that 
 \begin{equation}
  \nonumber
   I_{t_1}\geq C_1(\alpha S_{r_1}+\beta S_{r_3}),
  \end{equation}
 where  we have set \(C_1=\min{\left\{\frac{4}{r_1+2},\frac{4}{r_3+2},1\right\}}\).
\end{proof}
\[
\]
%

Next, we present a version of Proposition \ref{prope1}, where now, under stronger assumptions on the coefficients and exponents, we are able to consider certain modifications on the estimates which provide us alternative regularity results. This is possible because of the effects of $\beta \Delta^\epsilon_q u$ which allow slightly modified test functions. Also, now we provide extra degrees of fractional regularity which are a result from the interaction between the two differential operators.
\begin{proposition}
 \label{prope2}
Given $u\in W^{3,\tau}$, where $\tau>N$, suppose that \eqref{hip1}-\eqref{hip3} hold and that $\beta>0$. By setting \(t_2=s(q-2)+2-q,\) we obtain \(C>0\) such that
   \begin{equation*}
    I_{t_2}\geq C\left(S_{r_2}+S_{r_4}\right),
   \end{equation*}
where \(C=C(\alpha,\beta,p,q,r)\), \(r_2=s(q-2)+2\) and \(r_4=s(q-2)+2+p-q\).

\end{proposition}
\begin{proof}
 Note that \[(r_2\geq q>2 \mbox{ and } \frac{r_2+2}{4}\geq 1.\] Then, since \(t_2\geq 0\), by \eqref{eq3leme2} in Lemma~\ref{leme2} we obtain \(I^2_{q,\beta,t_2}\geq 0.\) Thus,
 \begin{align*}
   \frac{r_2+2}{4\beta}I_{q,\beta,t_2}&=\frac{r_2+2}{4\beta}\left(I^1_{q,\beta,t_2}+I^2_{q,\beta,t_2}+I^3_{q,\beta,t_2}\right)
   \\
   &\geq \frac{r_2+2}{4\beta}I^1_{q,\beta,t_2}+\frac{r_2+2}{4\beta}I^3_{q,\beta,t_2}\\
                                   &=\frac{(r_2+2)(r_2-2)}{4} \sum_{i=1}^N\int_{\Omega}\left(|\nabla u|^2+\epsilon\right)^\frac{r_2-4}{2}\left(\sum_{j=1}^Nu_{x_ix_j}u_{x_i}\right)^2\\
                                   &+\int_\Omega\left(|\nabla u|^2+\epsilon\right)^\frac{r_2-2}{2}|D^2u|\\
                                   &\geq S_{r_2},
 \end{align*}
where in the last inequality \eqref{eq3obs1} is used. Hence
  \begin{equation}
   \nonumber
     I_{q,\beta,t_2}\geq C_1 S_{r_2},
   \end{equation}
where \(\displaystyle C_1=\frac{4\beta}{r_2+2}\).

By applying similar arguments, we arrive at
    \( I_{p,\alpha,t_2}\geq C_2S_{r_4}\),
where \(C_2=\frac{4\alpha}{r_4+2}\). Thence, by choosing \(C=\min{\left\{C_1,C_2\right\}}\), from the last inequalities 
 we conclude that
  \begin{equation*}
    I_{t_2}=I_{p,\alpha,t_2}+I_{q,\beta,t_2}\geq C\left(S_{r_2}+S_{r_4}\right).\qedhere
  \end{equation*}
\end{proof}

We are now in position to prove our ultimate set of energy inequalities which bind the boundary estimates refinements done at the previous section and the energy inequalities obtained above, what will be sufficient to prove our main results.
We start by the Primary Energy Bounds, i.e., the estimates which are valid regardless of possible regularizing phenomena due to the influence of $\beta \Delta_q^\epsilon u$.
\begin{proposition}
\label{primaryfractional}
Consider $u$ the unique solution of \eqref{eq1ap}. Then, the following estimate holds true
\[ \|u\|_{\mathcal{N}^{1+\frac{2}{r_1},r_1}}^{r_1} \leq C \bigg (\|f\|^{\frac{r_1}{p-1}}_{W^{\sigma,s}}+\|f\|^{\frac{r_1}{p-1}}_{L^{p^\prime}}+\|f\|^s_{L^s} +\|f\|^2_{L^s}+\epsilon^{\frac{t_1+1}{2}} \|f\|_{W^{\sigma,s}} +\epsilon^{\frac{s(r_1-2)}{2}}\bigg) \]
where $C=C(N,p,s,\alpha,\beta,\sigma,\Omega)>0$.

\end{proposition}
\begin{proof}
As a starting point, remark that by taking $u$ as a test function in \eqref{eq1ap}, by the Poincar\'{e} inequality we clearly have
\begin{equation}
\label{eq1pf}
\|\nabla u\|_{L^p}\leq \dfrac{1}{\alpha ^{p-1}}\|f_\epsilon\|^{\frac{1}{p-1}}_{L^{p^\prime}}.
\end{equation}

Moreover, by Lemma \ref{leme1}, if $-\Delta_{t_1+2}^\epsilon u$ is taken as a test function, we get
\begin{equation}
\label{eq2pf}
I_{t_1}=I_{p,\alpha,t_1}+I_{q,\beta,t_1}=-\int_\Omega f_\epsilon \Delta^\epsilon_{t_1+2}u+F_{t_1},
\end{equation}
where $I_{p,\alpha,t_1}$ and $I_{q,\beta,t_3}$ were given in \eqref{efunc} and $F_{t_1}$, the boundary term, in \eqref{befunc}.

In addition, observe that by combining Lemmas \ref{lemb3} and \ref{lemb4}, \cite[Theorem 1.5.1.2, pag. 37]{grisvard} and  Young's inequality for $\frac{r_1}{t_1+1}$ and $\frac{r_1}{p-1}$, there follows that
\begin{align}
\nonumber
|F_{t_1}|\leq &C\|f_\epsilon\|_{L^s(\partial \Omega)} \bigg( \|\nabla u\|^{t_1+1}_{L^{s^\prime(t_1+1)}(\partial \Omega)}+\epsilon^{\frac{t_1+1}{2}}\bigg)
\\
\nonumber
&\leq C\|f_\epsilon\|_{W^{\sigma,s}( \Omega)}\bigg(\|u\|^{t_1+1}_{\mathcal{N}^{1+\frac{2}{r_1},r_1}(\Omega)}+\epsilon^{\frac{t_1+1}{2}}\bigg)
\\
\label{eq3pf}
&\leq \dfrac{C}{\eta^{\frac{r_1-p+1}{p-1}}}\|f\|^{\frac{r_1}{p-1}}_{W^{\sigma,s}}+\epsilon^{\frac{t_1+1}{2}} \|f\|_{W^{\sigma,s}}+\eta\|u\|^{r_1}_{\mathcal{N}^{1+\frac{2}{r_1},r_1}},
\end{align}
since $t_1=r_1-p$, where $C=C(N,p,s,\sigma,\Omega)>0$, and $\eta>0$.

At this point, we proceed to estimate
\[\int_\Omega -f_\epsilon \Delta^\epsilon_{t_1+2} u.\]
In this fashion, by setting
\[\omega_1= 2-p+\frac{r_1-2}{2},\]
remark that $\omega_1\geq0$, since $s\geq2$. Further, for a given $C=C(p,s)>0$,
\begin{align}
\label{eq4pf}
\int_\Omega -f_\epsilon \Delta^\epsilon_{t_1+2} u&\leq C \int_\Omega |f_\epsilon| \big( |\nabla u|^2+\epsilon\big)^{\frac{t_1}{2}}|D^2u|
\\
\nonumber
&= C \int_\Omega |f_\epsilon| \big( |\nabla u|^2+\epsilon\big)^{\frac{r_1-2}{4}}|D^2u|\big( |\nabla u|^2+\epsilon\big)^{\frac{\omega_1}{2}}
\end{align}
since $t_1=\frac{r_1-2}{2}+\omega_1$.

Suppose for instance that $s>2$ so that $w_1>0$.
At this point, for the sake of clarity, we are going to consider separately the cases
\begin{equation}
\label{eq5pf}
\|u\|_{\mathcal{N}^{1+\frac{2}{r_1},r_1}}\geq 1
\end{equation}
and
\begin{equation}
\label{eq6pf} \|u\|_{\mathcal{N}^{1+\frac{2}{r_1},r_1}}< 1,
\end{equation}
where, of course,  our purpose is to obtain a uniform estimate which holds for both cases.

Now, let us admit that \eqref{eq5pf} holds. Thence, by recalling that
\[\frac{s \omega_1}{s-2}=\frac{r_1-2}{2}\]
and by applying H\"{o}lder's and Young's inequalities for $s$, $2$ and $\frac{2s}{s-2}$  in the right-hand side of \eqref{eq4pf} we get
\begin{align*}
\int_\Omega -f_\epsilon \Delta^\epsilon_{t_1+2} u&\leq \frac{C}{\eta^{\frac{s}{s^\prime}}}  \|f_\epsilon\|^s_{L^s} +\eta\int_\Omega \big( |\nabla u|^2+\epsilon\big)^{\frac{r_1-2}{2}}|D^2u|^2
\\
&+\eta\int_\Omega \big( |\nabla u|^2+\epsilon\big)^{\frac{r_1-2}{2}}.
\end{align*}
In addition, by combining the H\"{o}lder inequality and \eqref{eq5pf}, it is clear that
\begin{align*}
\int_\Omega |\nabla u|^{r_1-2} &\leq \bigg( \int_\Omega |\nabla u|^{r_1}\bigg)^{\frac{r_1-2}{r_1}}|\Omega|^{\frac{2}{r_1}}
\\
&\leq C\|u\|^{r_1-2}_{\mathcal{N}^{1+\frac{2}{r_1},r_1}}
\\
&\leq C\|u\|^{r_1}_{\mathcal{N}^{1+\frac{2}{r_1},r_1}},
\end{align*}
where $C=C(p,s,\Omega)>$.

Further, by recalling the definition of $S_{r_1}$, it is clear that
\[ \int_\Omega \big( |\nabla u|^2+\epsilon\big)^{\frac{r_1-2}{2}}|D^2u|^2 \leq S_{r_1},\]
see Lemma \ref{leme3}.

Thence, by the latter inequalities, there follows
\begin{equation}
\label{eq7pf}
-\int_\Omega f_\epsilon \Delta^\epsilon_{t_1+2} u \leq C \bigg (\frac{\|f\|^s_{L^s}}{\eta^{\frac{s}{s^\prime}}}+\eta \epsilon^{\frac{s(r_1-2)}{2}} + \eta \|u\|^{r_1}_{\mathcal{N}^{1+\frac{2}{r_1},r_1}}+\eta S_{r_1} \bigg),
\end{equation}
for $C=C(p,s,\Omega)>0$.

In the sequel, let us assume the validity of \eqref{eq6pf}, what in particular assures that
\[\int_\Omega |\nabla u|^{r_1} \leq C,\]
since $\mathcal{N}^{1+\frac{2}{r_1},r_1}\hookrightarrow W^{1,r_1}$, regardless of our choice of norm for this Nikolskii space. Before going any further, we must stress that our goal is to obtain an estimate for $\|u\|_{\mathcal{N}^{1+\frac{2}{r_1},r_1}}$ in terms of $f$. Thereby, from \eqref{eq4pf} and the definition of $S_{r_1}$, we arrive at
\begin{align}
\nonumber
\int_\Omega -f_\epsilon \Delta_{t_1+2}^\epsilon u  &\leq C \|f_\epsilon\|_{L^s}\bigg(\int_\Omega \big( |\nabla u|^2+\epsilon\big)^{\frac{r_1-2}{2}}|D^2u|^2\bigg)^{\frac{1}{2}}\bigg(\int_\Omega \big( |\nabla u|^2+\epsilon\big)^{\frac{r_1-2}{2}}\bigg)^{\frac{s-2}{2s}}
\\
\nonumber
&\leq C \|f_\epsilon\|_{L^s} S_{r_1}^{\frac{1}{2}}\bigg(\int_\Omega  |\nabla u|^{r_1-2}+|\Omega|\bigg)^{\frac{s-2}{2s}}
\\
\nonumber
&\leq C \|f_\epsilon\|_{L^s} S_{r_1}^{\frac{1}{2}}\bigg(\|\nabla u\|_{L^{r_1}}^{\frac{r_1-2}{2}}+|\Omega|\bigg)^{\frac{s-2}{2s}}
\\
\label{eq8pf}
&\leq \dfrac{C}{\eta^2} \|f\|^2_{L^s} +\eta S_{r_1},
\end{align}
for $C=C(N,p,s,\Omega)>0$, where we employed H\"{o}lder's inequality twice, for the triple $2$, $s$,$\frac{2s}{s-2}$, and for the couple  $\frac{r_1}{r_1-2}$,$\frac{r_1}{2}$. Remark that $S_{r_1}$ absorbs the $\epsilon$'s term.

Hence, by combining \eqref{eq7pf} and \eqref{eq8pf}, we obtain
\begin{align}
\nonumber
\int_\Omega -f_\epsilon \Delta_{t_1+2}^\epsilon u  &\leq C\bigg(\frac{\|f\|^s_{L^s}}{\eta^{\frac{s}{s^\prime}}}+\dfrac{\|f\|^2_{L^s}}{\eta^2}  + \eta \|u\|^{r_1}_{\mathcal{N}^{1+\frac{2}{r_1},r_1}} +2\eta S_{r_1}+\eta \epsilon^{\frac{s(r_1-2)}{2}} \bigg)
\end{align}
and thus, by combining \eqref{eq2pf} and \eqref{eq3pf}, there follows that
\begin{align*}I_{t_1}-2\eta S_{r_1} -2\eta\|u\|^{r_1}_{\mathcal{N}^{1+\frac{2}{r_1},r_1}}&\leq C\bigg(\dfrac{\|f\|^{\frac{r_1}{p-1}}_{W^{\sigma,s}}}{\eta^{\frac{r_1-p+1}{p-1}}}+\frac{\|f\|^s_{L^s}}{\eta^{\frac{s}{s^\prime}}}+\dfrac{\|f\|^2_{L^s}}{\eta^2}  \\
&+\epsilon^{\frac{t_1+1}{2}} \|f\|_{W^{\sigma,s}}+\eta \epsilon^{\frac{s(r_1-2)}{2}}\bigg).
\end{align*}
However, since, by Lemma \ref{lemp5} and Poincar\'{e}'s inequality
\[S_{r_1}\geq C\bigg(\|u\|^{r_1}_{\mathcal{N}^{1+\frac{2}{r_1},r_1}}-\|\nabla u\|^{r_1}_{L^p}\bigg),\]
and Proposition \ref{prope1},
\[ I_{t_1}\geq C S_{r_1},\]
it is clear from the latter inequalities that, for $\eta>0$ sufficiently small,
\[ \|u\|_{\mathcal{N}^{1+\frac{2}{r_1},r_1}}^{r_1} \leq C \bigg (\|f\|^{\frac{r_1}{p-1}}_{W^{\sigma,s}}+\|f\|^{\frac{r_1}{p-1}}_{L^{p^\prime}}+\|f\|^s_{L^s} +\|f\|^2_{L^s}+\epsilon^{\frac{t_1+1}{2}} \|f\|_{W^{\sigma,s}} +\epsilon^{\frac{s(r_1-2)}{2}}\bigg), \]
where we have used \eqref{eq1pf}.

The case $s=2$ is absolutely analogous, where the only difference is that, since in this case $w_1=0$ we are able to obtain our estimates for $\int_\Omega -f_\epsilon \Delta^\epsilon_{t_1+2} $ directly from \eqref{eq4pf}.
\end{proof}
Under additional assumptions, it is possible to enlarge our set of estimates what allows us to take into account other levels of regularity for solutions of \eqref{eq1ap} or \eqref{eq1}, the limit case, which appear due to the interactions between the degenerated operators.
\begin{proposition}
\label{secondaryfractional}
Consider $u$ the unique solution of \eqref{eq1ap} and suppose that $\beta>0$.
\begin{itemize}
\item[(a)] For the pair $(r_2,t_2)$, the following estimate holds
\begin{equation}
\nonumber \|u\|_{\mathcal{N}^{1+\frac{2}{r_2},r_2}}^{r_2} \leq C \bigg (\|f\|^{\frac{r_2}{q-1}}_{W^{\sigma,s}}+\|f\|^{\frac{r_2}{q-1}}_{L^{q^\prime}}+\|f\|^s_{L^s} +\|f\|^2_{L^s}+\epsilon^{\frac{t_2+1}{2}} \|f\|_{W^{\sigma,s}} +\epsilon^{\frac{s(r_2-2)}{2}}\bigg).
\end{equation}
\item[(b)] If $s\geq \frac{q+p-4}{p-2}$ , then
\[
\nonumber \|u\|_{\mathcal{N}^{1+\frac{2}{r_3},r_3}}^{r_3} \leq C \bigg (\|f\|^{\frac{r_3}{q-1}}_{W^{\sigma,s}}+\|f\|^{\frac{r_3}{q-1}}_{L^{q^\prime}}+\|f\|^{\tau}_{L^s} +\|f\|^2_{L^s}+\epsilon^{\frac{t_3+1}{2}} \|f\|_{W^{\sigma,s}} +\epsilon^{\frac{s(r_3-2)}{2}}\bigg)
\]
where
\[\tau= \frac{s(p-2)+q-p}{q-2}.\]
\item[(c)]Further, if
\[p<q<p+1 \mbox{ and }s\geq 1 + \dfrac{1}{1+p-q}\]
then, there holds that
\[ \|u\|_{\mathcal{N}^{1+\frac{2}{r_4},r_4}}^{r_4} \leq C \bigg (\|f\|^{\frac{r_4}{p-1}}_{W^{\sigma,s}}+\|f\|^{\frac{r_4}{p-1}}_{L^{p^\prime}}+\|f\|^{\rho}_{L^{\rho}} +\|f\|^2_{L^s}+\epsilon^{\frac{t_4+1}{2}} \|f\|_{W^{\sigma,s}} +\epsilon^{\frac{s(r_4-2)}{2}}\bigg) \]
where \[\rho= \frac{s(q-2)+p-q}{p-2}.\]
\end{itemize}
In all the latter cases, $C=C(N,p,q,s,\alpha,\beta,\sigma,\Omega)>0$.
\end{proposition}
\begin{proof}

First of all, remark that, from the definition of $\mathcal{N}^{1+\frac{2}{r_i},r_i}$, for $i=2$ or $3$, we have to estimate $u$ in $W^{1,q}$. However, since $\beta>0$, by taking $u$ as a test function we have
\begin{equation}
\label{eq1sf}
\|\nabla u\|_{L^q}\leq \dfrac{1}{\beta ^{q-1}}\|f_\epsilon\|^{\frac{1}{q-1}}_{L^{q^\prime}}.
\end{equation}

\subsection*{Proof of (a)} Since $r_2=s(q-2)+2$ and $t_2=r_2-p$, this proof is entirely analogous to the proof of Proposition \ref{primaryfractional}, and is left to the reader. Remark that here we apply Proposition \ref{prope2}, instead of Proposition \ref{prope1}.

\subsection*{Proof of (b)} In this case, as $r_3=s(p-2)+q-p+2$ and $t_3=r_3-p$, there are certain minor modifications in the proof.

Indeed, once again we apply Lemma \ref{leme1} and take  $-\Delta_{t_3+2}^\epsilon u$ as a test function, so that
\begin{equation}
\label{eq2sf}
I_{t_3}=I_{p,\alpha,t_3}+I_{q,\beta,t_3}=-\int_\Omega f_\epsilon \Delta^\epsilon_{t_3+2}u+F_{t_3},
\end{equation}
where the functionals were given in \eqref{efunc} and the boundary term, in \eqref{befunc}.

Moreover, by the very same argument given in \eqref{eq3pf}, we get
\begin{align}
\nonumber
|F_{t_3}|\leq &C\|f_\epsilon\|_{L^s(\partial \Omega)} \bigg( \|\nabla u\|^{t_3+1}_{L^{s^\prime(t_3+1)}(\partial \Omega)}+\epsilon^{\frac{t_3+1}{2}}\bigg)
\\
\label{eq3sf}
&\leq \dfrac{C}{\eta^{\frac{r_3-q+1}{q-1}}}\|f\|^{\frac{r_3}{q-1}}_{W^{\sigma,s}}+\epsilon^{\frac{t_3+1}{2}} \|f\|_{W^{\sigma,s}}+\eta\|u\|^{r_3}_{\mathcal{N}^{1+\frac{2}{r_3},r_3}}.
\end{align}

Now, in order to handle
\[\int_\Omega -f_\epsilon \Delta^\epsilon_{t_3+2} u\]
we consider
\[\omega_4= 2-q+\frac{r_3-2}{2},\]
where since $s\geq \frac{q+p-4}{p-2}$, we have $\omega_2\geq 0$.

In this fashion, for a given $C=C(p,s)>0$,
\begin{align}
\label{eq4sf}
\int_\Omega -f_\epsilon \Delta^\epsilon_{t_3+2} u&\leq C \int_\Omega |f_\epsilon| \big( |\nabla u|^2+\epsilon\big)^{\frac{t_3}{2}}|D^2u|
\\
\nonumber
&= C \int_\Omega |f_\epsilon| \big( |\nabla u|^2+\epsilon\big)^{\frac{r_3-2}{4}}|D^2u|\big( |\nabla u|^2+\epsilon\big)^{\frac{\omega_2}{2}}
\end{align}
since $t_3=\frac{r_3-2}{2}+\omega_2$.

However, on one hand, if $\omega_2=0$, or equivalently  $s=\frac{q+p-4}{p-2}$, observe that the combination between the Cauchy-Schwartz and the Young inequalities gives us
\begin{equation}
\label{eq44sf}
\int_\Omega -f_\epsilon \Delta^\epsilon_{t_3+2} u \leq \frac{C}{\eta^2}  \|f_\epsilon\|^2_{L^2} +\eta\int_\Omega \big( |\nabla u|^2+\epsilon\big)^{\frac{r_3-2}{2}}|D^2u|^2.
\end{equation}
On the other hand, if $\omega_2>0$, or equivalently  $s>\frac{q+p-4}{p-2}$, we once again have to deal with the cases
\begin{equation}
\nonumber
\|u\|_{\mathcal{N}^{1+\frac{2}{r_3},r_3}}\geq 1 \mbox{ and } \|u\|_{\mathcal{N}^{1+\frac{2}{r_3},r_3}}< 1
\end{equation}
separately. For this, in order to replicate our previous argument, we are interested in applying H\"{o}lder's inequality to \eqref{eq4sf}, for $\tau$, $2$ and $\frac{2\tau}{\tau-2}$, where $\tau>2$ is a given exponent, so that
\[\frac{\tau \omega_2}{\tau-2}=\frac{r_3-2}{2}.\]

Thus, consider
\[\tau= \frac{s(p-2)+q-p}{q-2}\]
and remark that if hypothetically $q=p$, we would have $\tau=s$, which was the choice in the proof on Proposition \ref{primaryfractional}.
Moreover, observe that $\tau>2$ since $s>\frac{q+p-4}{p-2}$.

In this manner, let us consider $\|u\|_{\mathcal{N}^{1+\frac{2}{r_3},r_3}}\geq 1$, so that, by combining H\"{o}lder's and Young's inequalities, from \eqref{eq4sf} we have
\begin{align}
\label{eq444sf}
\int_\Omega -f_\epsilon \Delta^\epsilon_{t_3+2} u&\leq C\bigg(\frac{1}{\eta^{\frac{\tau}{\tau^\prime}}}  \|f_\epsilon\|^\tau_{L^\tau} +\eta\int_\Omega \big( |\nabla u|^2+\epsilon\big)^{\frac{r_3-2}{2}}|D^2u|^2
\\
\nonumber
&+\eta\int_\Omega \big( |\nabla u|^2+\epsilon\big)^{\frac{r_3-2}{2}}\bigg)
\\
\nonumber
&\leq C\bigg(\frac{1}{\eta^{\frac{\tau}{\tau^\prime}}}  \|f_\epsilon\|^\tau_{L^\tau} +\eta S_{r_3}+\eta\|u\|^{r_3}_{\mathcal{N}^{1+\frac{2}{r_3},r_3}}\bigg)
\end{align}
for $C=C(N,p,q,s,\Omega)>0$. We stress that in \eqref{eq444sf}, it is used that
\[ \int_\Omega \big( |\nabla u|^2+\epsilon\big)^{\frac{r_3-2}{2}} \leq S_{r_3},\]
and
\begin{align*}
\int_\Omega |\nabla u|^{r_3-2} &\leq \bigg( \int_\Omega |\nabla u|^{r_3}\bigg)^{\frac{r_3-2}{r_3}}|\Omega|^{\frac{2}{r_3}}
\\
&\leq C\|u\|^{r_3-2}_{\mathcal{N}^{1+\frac{2}{r_3},r_3}}
\\
&\leq C\|u\|^{r_3}_{\mathcal{N}^{1+\frac{2}{r_3},r_3}}.
\end{align*}

For the case where $\|u\|_{\mathcal{N}^{1+\frac{2}{r_3},r_3}}< 1$, we once again will have
\[\int_\Omega |\nabla u|^{r_3} \leq C.\]

Thus, by the same argument given in \eqref{eq8pf}, we arrive at

\begin{align}
\label{eq5sf}
\int_\Omega -f_\epsilon \Delta_{t_3+2}^\epsilon u  &\leq \dfrac{C}{\eta^2} \|f\|^2_{L^\tau} +\eta S_{r_3},
\end{align}
for $C=C(N,p,q,s,\Omega)>0$.

In this fashion, since $2<\tau<s$, by combining inequalities \eqref{eq44sf},\eqref{eq444sf} and \eqref{eq5sf}, we obtain
\begin{align}
\nonumber
\int_\Omega -f_\epsilon \Delta_{t_3+2}^\epsilon u  &\leq C\bigg(\frac{\|f\|^\tau_{L^s}}{\eta^{\frac{s}{s^\prime}}}+\dfrac{\|f\|^2_{L^s}}{\eta^2}  + \eta \|u\|^{r_3}_{\mathcal{N}^{1+\frac{2}{r_3},r_3}} +2\eta S_{r_3}+\eta \epsilon^{\frac{s(r_3-2)}{2}} \bigg)
\end{align}
and thus, by combining the latter inequality, \eqref{eq1sf}, \eqref{eq2sf} and \eqref{eq3sf}, with the same argument of Proposition \ref{primaryfractional}, the result follows.

\subsection*{Proof of (c)}



we consider
\[\omega_4= 2-p+\frac{r_4-2}{2},\]
where since $s> \frac{q+p-4}{q-2}$, we have $\omega_4> 0$.

Indeed, remark that as $p<q<p+1$ then $s\geq 1+\frac{1}{1+p-q}>\frac{q+p-4}{q-2}$.

since $t_4=\frac{r_4-2}{2}+\omega_4$.

\[\frac{\rho \omega_4}{\tau-2}=\frac{r_4-2}{2}.\]

Thus, consider
\[\rho= \frac{s(q-2)+p-q}{p-2}\]
and remark that if hypothetically $q=p$, we would have $\rho=s$, which was the choice in the proof on Proposition \ref{primaryfractional}.
Moreover, observe that $\rho>2$ since $s>\frac{q+p-4}{q-2}$.

Actually, in this case, $\rho>s\geq 2$, which is the main difference when contrasting to the last case.

\end{proof}

\section{Proof of the main results}
\label{mainresults}

First, we handle the proof of Theorem \eqref{theorem1} where the effects of the nonlinear differential operator of order $q$ may play no interference in the regularity of the solutions.


\subsection*{Proof of Theorem \eqref{theorem1}}
\vspace{0.5cm}

\begin{itemize}
\item[ {\bf Step 1.}] Existence and Regularity of the Strong Solution.
\end{itemize}
Set $\epsilon=\frac{1}{n}$ and consider $u_n$ given by Lemma \ref{approximatesolutions}.

Here, for the sake of clarity, we shall deal with the cases $\beta=0$ and $\beta>0$.
Thus, for the moment, let us assume that $\beta=0$.

Then, it is clear that by Proposition \ref{primaryfractional}, there exists $u\in W^{1,p}_0\cap \mathcal{N}^{1+\frac{2}{r_1},r_1}$ for which, up to subsequences, if $n\rightarrow+\infty$
\[ u_n \rightarrow u \mbox{ in } W^{1,p}_0 \mbox{ and }  u_n \rightharpoonup u \mbox{ in } \mathcal{N}^{1+\frac{2}{r_1},r_1},\]
where we stress that
\[ \mathcal{N}^{1+\frac{2}{r_1},r_1} \hookrightarrow\hookrightarrow W^{1,r_1} \hookrightarrow W^{1,p}_0.\]
Then it is clear that $u$ satisfies
\[\int_\Omega \alpha |\nabla u|^{p-2}\nabla u \cdot \nabla \phi = \int_\Omega f \phi \ \forall \phi \in W^{1,p}_0,\]
what in particular assures that the distribution associated to \( \mbox{div} \bigg (|\nabla u|^{p-2}\nabla u \bigg)\) belongs to \( L^2.\) Moreover, by a generalized integration by parts formula
\[\int_\Omega |\nabla u|^{p-2}\nabla u \cdot \nabla \phi = \int_\Omega -\Delta_p u \phi, \forall \phi \in W^{1,p}_0\]
since
\(|\nabla u|^{p-2}\nabla u \in L^{p^\prime}\), see Theorem III.2.2 and identity III.2.4 in \cite{galdi}, pp. 159 and 160.

Thus, by the latter identities, we have that
\[-\alpha \Delta_p u = f \mbox{ a.e in } \Omega.\]
Suppose now that $\beta>0$. Then, by Proposition \ref{secondaryfractional}, item (a), we would also have that, up to subsequences
\[ u_n \rightarrow u \mbox{ in } W^{1,q}_0 \mbox{ and }  u_n \rightharpoonup u \mbox{ in } \mathcal{N}^{1+\frac{2}{r_2},r_2},\]
so that by the same arguments presented above  we obtain that $u\in W^{1,q}_0\cap \mathcal{N}^{1+\frac{2}{r_2},r_2}$ and
\[-\alpha \Delta_p u -\beta \Delta_q u  = f \mbox{ a.e in } \Omega,\]
while in this case, we stress that the test functions should belong to $W^{1,q}_0$.

\begin{itemize}
\item[ {\bf Step 2.}] A priori estimate.
\end{itemize}

 Now let $\beta \geq 0$.
By recalling that $\epsilon=\frac{1}{n}$, in order to prove the energy estimate, it is enough to let $n\rightarrow +\infty$ in the energy estimate given by Proposition \ref{primaryfractional} and employ the lower semicontinuity of the norm.
\hfill\qedsymbol

\subsection*{Proof of Theorem \eqref{theorem2}}
In order to proof that  $u\in W^{1,q}_0\cap \mathcal{N}^{1+\frac{2}{r_2},r_2}$ and
\[-\alpha \Delta_p u -\beta \Delta_q u  = f \mbox{ a.e in } \Omega,\] it is enough to repeat the argument given in the second part of Step 1 in the proof of Theorem \ref{theorem1}.
For the energy estimate, in this case, we pass to the limit at the estimate given by Proposition \ref{secondaryfractional} item (a).
\hfill\qedsymbol

\subsection*{Proof of Theorem \eqref{theorem3}}
\begin{itemize}
\item[ {\bf (a)}] We consider once again the approximate solutions $u_n \in W^{3,\tau}$ given by Lemma \ref{approximatesolutions}, where $\epsilon =\frac{1}{n}$. Then, by combining the latter arguments with Proposition \ref{secondaryfractional}, item (b), it is enough to let $n\rightarrow +\infty$.
\end{itemize}
\begin{itemize}
\item[ {\bf (b)}] In this case, we employ item (c) of Proposition \ref{secondaryfractional} and mimic the proof of item (a).
\end{itemize}

\hfill\qedsymbol

\subsection*{Acknowledgments}
The authors would like to thank the  anonymous referee for all of her/his  suggestions which have helped to improve this paper.


\end{document}